%% file: acoustic5.tex
\documentclass[11pt]{article}
\usepackage{amsfonts,amsmath,amssymb,amsthm,graphicx}
\usepackage{multirow}
\usepackage{booktabs}
\usepackage[caption=false]{subfig} 
\usepackage{color}
\usepackage{ifthen}
\usepackage{tabu}
\usepackage{tabularx}
\usepackage[hidelinks]{hyperref}

\newtheorem{lemma}{Lemma}

\newtheorem{assumption}{Assumption}

\include{macros}

\begin{document}

\title{Energy conservative SBP discretizations of the acoustic wave equation in covariant
form on staggered curvilinear grids}

\author{Ossian O'Reilly\thanks{Southern California Earthquake Center,
University of Southern California, Los Angles, CA 90089.}
  \and N. Anders Petersson\thanks{Center for Applied
    Scientific Computing, Lawrence Livermore National Laboratory, L-561, PO Box 808, Livermore CA
    94551.}
}

\date{\today}

\maketitle
\begin{abstract}
We develop a numerical method for solving the acoustic wave equation in covariant form on staggered
curvilinear grids in an energy conserving manner. The use of a covariant basis decomposition leads
to a rotationally invariant scheme that outperforms a Cartesian basis decomposition on rotated
grids. The discretization is based on high order Summation-By-Parts (SBP) operators and preserves
both symmetry and positive definiteness of the contravariant metric tensor. To improve accuracy and
decrease computational cost, we also derive a modified discretization of the metric tensor that
leads to a conditionally stable discretization. Bounds are derived that yield a point-wise condition
that can be evaluated to check for stability of the modified discretization. This condition shows that
the interpolation operators should be constructed such that their norm is close
        to \revisedred{one}.
\end{abstract}

\section{Introduction}
Many applications featuring wave propagation require numerical methods capable of treating complex
geometry. A popular approach for finite difference methods is to solve the governing equations on
curvilinear grids. 
Typically, the governing equations are first formulated in terms of Cartesian
vector components followed by transforming the spatial derivatives with respect
to the curvilinear coordinates. An alternative approach utilizes the invariance of the governing equation with
respect to the choice of coordinate basis. In this paper we develop a staggered grid discretization
of the acoustic wave equation where the components of the velocity vector are expressed with respect
to the covariant basis, which varies throughout space.

Discretization of wave equations on general staggered curvilinear grids is challenging due to the
emergence of off-diagonal metric terms. On a Cartesian grid, all fields and their components can be
positioned in the grid so that all of the terms can be discretized by staggered difference
approximations. However, a curvilinear coordinate transformation introduces additional terms into
the governing equations that cannot be discretized without interpolation. One can directly
discretize these additional terms by collocating some of the components, but this collocation
increases the computational cost and memory storage. Despite this increase in computational cost,
the approach has been used in computational seismology for solving the elastic wave equation
\cite{de2014mimetic, tarrass2011new, PerezSolano2016}. An alternative solution is to solve the
covariant form of the governing equations on orthogonal grids. In this case, the metric tensor
becomes diagonal. This approach has been used in computational electrodynamics for solving Maxwell's
equations in covariant form \cite{xie2002explicit, taflove2005computational}.  Unfortunately, the
generation of orthogonal grids is a non-trivial task, in particular for many practical applications
that feature complex 3D geometries. A third solution is to use interpolation in the off-diagonal
components. So far this approach has only been developed with one-dimensional low order
interpolation~\cite{hestholm19942d, hestholm19983, doi:10.1046/j.1365-2478.2002.00327.x}, resulting
in a scheme that may be susceptible to long-term instabilities, requiring ad-hoc
stabilization~\cite{doi:10.1190/1.1543217}. In the present work, we develop a provable stable
numerical method that solves the acoustic wave equation in covariant form on a general staggered
curvilinear grids. 
The proposed scheme is high-order-accurate, energy conservative and preserves
positive definiteness of the discretized metric tensor.

To obtain a provably stable method, we discretize the acoustic wave equation by applying the
Summation-By-Parts (SBP) principle. The SBP principle provides a way to analyze and derive provably
stable schemes by constructing energy estimates for semi-discrete approximations. SBP methods were
originally designed to obtain provably stable high order finite difference approximations for
first order hyperbolic systems on collocated grids~\cite{Kreiss-Scherer-74, Strand-94, Olsson-95a,
  Olsson-95b}. Since that time, there have been numerous developments that, for example, extend the
SBP methodology to non-conforming
grids~\cite{mattsson2010stable,petersson-sjogreen-meshref,gao2019sbp} and to equations with second
and higher order
derivatives~\cite{Sjogreen-Petersson-12,Petersson-Sjogreen-aniso-15,mattsson2014diagonal}.  SBP
methods have also been extend beyond finite differences, e.g., to finite volume, discontinuous
Galerkin, and flux reconstruction schemes~\cite{nordstrom2003finite,
  gassner2013skew,ranocha2016summation}. Methods have also been devised for coupling different
SBP discretizations~\cite{nordstrom2006stable,lundquist2018hybrid,kozdon2016stable,gao2019combining}.

Our proposed curvilinear staggered SBP discretization builds upon our previous work on staggered
grid methods for wave equations in Cartesian
geometries~\cite{o2017energy,prochnow2017treatment,mattsson2017compatible}. Our main theoretical
result is that stability can be guaranteed on a curvilinear grid under two conditions. First, the
interpolation operators must be compatible with the SBP norm weights such that the scalar product
between a staggered and a node-centered grid function gives the same result,
regardless of which
function is being interpolated. Secondly, the discretization must preserve the positive definiteness
of the metric tensor. The latter condition is essential because the velocity field is decomposed
with respect to the covariant basis. As a result, the symmetric and positive definite covariant
metric tensor appears in the kinetic energy term. The discretization of this tensor involves
interpolation, which can destroy positive definiteness. We show how to preserve positive
definiteness of this tensor in the discretized equations. While it is possible to preserve positive
definiteness for any non singular curvilinear grid, such a discretization has to involve interpolation in the
diagonal components of the contravariant metric tensor, which reduces accuracy and increases
computational cost. By modifying the metric tensor discretization in its diagonal components, it is
possible to overcome this reduction in accuracy and increase in computational cost. However,
positive definiteness can no longer be guaranteed for all \revisedgreen{non singular} grids.  We derive a condition that can be
checked in a point-wise manner to ensure positive definiteness of the metric tensor
discretization.  This condition shows that the interpolation operators should be constructed such
that their norm is close to \revisedred{one}.

The remainder of the paper is organized as follows. The covariant form of the acoustic wave equation
and curvilinear mappings are reviewed in Section \ref{s:problem}.  In Section \ref{s:operators} we
develop SBP finite difference for staggered grids, first in one dimension and then in two dimensions
for the acoustic wave equation. In Section \ref{s:energy}, we show that the discretization is energy
conservative.  Two alternative ways of discretizing the metric tensor such that the positive
definiteness is preserved are presented in Section \ref{s:contravariant}.  In Section
\ref{s:experiments}, we conduct numerical experiments. We first investigate the accuracy of the
scheme for the two different metric tensor discretizations. We then demonstrate that the covariant
formulation can outperform a Cartesian formulation due to the loss of
rotation\revisedgreen{al} invariance in the
latter. The section is concluded by demonstrating that the solution is free from numerical artifacts
when the acoutic wave propagation is driven by a discretized point force applied to the boundary. Finally, in
Section \ref{s:conclusions}, the study is summarized and conclusions are drawn.

\section{Problem formulation}\label{s:problem}
Consider the 2-D acoustic wave equation in dimensionless form (all quantities are scaled, including
space and time)
\begin{align}
        \frac{\p p}{\p t} + \nabla \cdot \boldsymbol{v} = 0, 
        \label{vec_pressure}
        \\
        \frac{\p \boldsymbol{v}}{\p t} + \nabla p = 0,
        \label{vec_velocity}
\end{align}
posed on a 2-D bounded domain $(x,y) \in \Omega$ for $t\geq 0$. Here, $p$ is the pressure and
$\boldsymbol{v}$ is the velocity vector. We consider cases where the domain $\Omega$ can be defined
through a curvilinear mapping from the unit square in parameter space \cite{ThoWarMas85}.  We define
the curvilinear coordinates
\[
0 \leq r^1 \leq 1 
\quad
\mbox{and}
\quad
0 \leq r^2 \leq 1,
\]
and the continuously differentiable mapping $x = X(r^1,r^2)$ and $y = Y(r^1,r^2)$ between the
curvilinear and Cartesian coordinates. \revisedgreen{We assume that the mapping
is non singular}. By differentiating the mapping with respect to the
curvilinear coordinates, one obtains the \emph{covariant} basis vectors,
\begin{align*}
        \ab_1 = 
        \begin{bmatrix}
                \frac{\partial X}{\partial r^1} \\
                \frac{\partial Y}{\partial r^1}
        \end{bmatrix}
        \quad
        \mbox{and}
        \quad
        \ab_2 = 
        \begin{bmatrix}
        \frac{\partial X}{\partial r^2} \\
        \frac{\partial Y}{\partial  r^2}
        \end{bmatrix}.
\end{align*}
The corresponding covariant metric tensor is defined by $g_{ij} = \ab_i \cdot
\ab_j$. It is symmetric and positive definite.

The \emph{contravariant} basis vectors $\ab^1$ and $\ab^2$ can be defined by the
orthogonality condition
\[
\ab^i \cdot \ab_j = \delta^i_j 
= 
\begin{cases}
        1, & i = j \\
        0, & i \ne j
\end{cases}.
\label{contravariant_basis}
\]
The contravariant metric tensor satisfies $g^{ij} = \ab^i\cdot \ab^j$. It is symmetric and positive
definite, and the Jacobian of the curvilinear mapping, $J=1/\sqrt{|g^{ij}|}$, is assumed to be
bounded. Here, $|g^{ij}|$ denotes the determinant of the tensor $g^{ij}$.

The velocity field is decomposed with respect to the covariant basis,
\[
        \boldsymbol{v} = v^1 \ab_1 + v^2 \ab_2,
\]
where $v^1$ and $v^2$ are the contravariant velocity components. The
contravariant metric tensor can be used to transform between covariant and
contravariant velocity components,
\begin{align}
        \begin{bmatrix}
        v^1 \\
        v^2
        \end{bmatrix}
        = 
        \begin{bmatrix}
                g^{11} & g^{12} \\
                g^{12} & g^{22}
        \end{bmatrix}
        \begin{bmatrix}
                v_1 \\
                v_2
        \end{bmatrix}.
        \label{eq:contravariant}
\end{align}
The inverse relationship is
\begin{align}
        \begin{bmatrix}
        v_1 \\
        v_2
        \end{bmatrix}
        = 
        \begin{bmatrix}
                g_{11} & g_{12} \\
                g_{12} & g_{22}
        \end{bmatrix}
        \begin{bmatrix}
                v^1 \\
                v^2
        \end{bmatrix}. 
        \label{eq:covariant}
\end{align}
Contravariant quantities are always denoted by superscripts, whereas covariant quantities are
always denoted by subscripts.

Next, we derive the covariant formulation of the acoustic wave equation. The divergence of the
velocity vector field satisfies \cite{ThoWarMas85, grinfeld2013introduction}
\begin{align}
        \nabla \cdot \boldsymbol{v} = \frac{1}{J}\frac{\p Jv^1}{\p 
        r^1} + \frac{1}{J}\frac{\p Jv^2}{\p r^2}.
\label{div}
\end{align}
The pressure gradient satisfies
\begin{align}
        \nabla p = \ab^1 \frac{\p p}{\p r^1} + \ab^2\frac{\p p}{\p r^2}. 
\label{grad}
\end{align}
By inserting (\ref{div}) into (\ref{vec_pressure}) and (\ref{grad}) into
(\ref{vec_velocity}), followed by taking the dot product of (\ref{vec_velocity})
with the contravariant basis, results in
\begin{align}
        J\frac{\p p}{\p t} + \frac{\p}{\p r^1}\left(Jv^1 \right) + \frac{\p}{\p
        r^2 }\left(Jv^2\right) &= 0,\label{eq:pressure}\\
\frac{\p }{\p t}
        \begin{bmatrix}
        v^1 \\
        v^2
        \end{bmatrix}
        +
        \begin{bmatrix}
                g^{11} & g^{12} \\
                g^{12} & g^{22}
        \end{bmatrix}
        \begin{bmatrix}
        \frac{\p p}{\p r^1} \\
        \frac{\p p}{\p r^2}
        \end{bmatrix}
                 &= \begin{bmatrix}
                         0 \\
                         0
                        \end{bmatrix}
                        .\label{eq:velocity}
\end{align}

The total energy in the system is the sum of the acoustic and the kinetic
energies 
\begin{align}
        e(t) = 
        \frac{1}{2}\int p^2 J dr^1dr^2 
        + \frac{1}{2}\int ( v_1v^1 + v_2v^2)
        J dr^1 dr^2.
\end{align}
The second integral expresses the kinetic energy as an invariant formed by contracting the covariant
velocity components with the contravariant components. Alternately, the covariant velocity
components can be transformed into the contravariant velocity components using
(\ref{eq:covariant}), yielding
\begin{align}
        e(t) = 
        \frac{1}{2}\int p^2 J dr^1dr^2 
        + \frac{1}{2}
        \sum_{i,j}
        \int 
        g_{ij}v^i  v^j J dr^1 dr^2.
        \label{continuous_energy}
\end{align}
In this form, we see that (\ref{continuous_energy}) is positive since the
covariant metric tensor $g_{ij}$ is symmetric and positive definite.  By
differentiating the energy (\ref{continuous_energy}) with respect to time and
inserting the governing equations
(\ref{eq:pressure})-(\ref{eq:velocity}), we obtain the energy rate
\begin{align*}
        \frac{de}{dt} &= 
        -\int p \left(
        \frac{\partial (Jv^1)}{\partial r^1} 
        +
        \frac{\partial (Jv^2)}{\partial r^2} 
        \right) dr^1 dr^2
        -
        \int 
        \left(
        \left( J v^1 \frac{\partial p}{\partial r^1}\right)
        +
        \left( J v^2 \frac{\partial p}{\partial r^2}\right)
        \right) dr^1 dr^2 \\
        &= 
        -\int \left(
        \frac{\partial (pJv^1)}{\partial r^1} 
        +
        \frac{\partial (pJv^2)}{\partial r^2} 
        \right) dr^1 dr^2 \\
        &= 
        -\int \left[ p J v^1 \right]_{r^1=0}^1 dr^2
        -\int \left[ p J v^2 \right]_{r^2=0}^1 dr^1.
\end{align*}
The terms in brackets on the right hand side must be bounded by imposing one boundary condition per
side of the unit square in parameter space. With the homogeneous pressure boundary condition $p = 0$
on all sides,
\[
        \frac{de}{dt} = 0.
\]
As expected, the acoustic wave equation in covariant formulation conserves the total energy in the system.

\section{SBP operators}\label{s:operators}
Before we can present the numerical scheme, we need to introduce some
definitions. First, we review SBP staggered grid finite difference
methods in one spatial dimension, and then in two spatial dimensions
\cite{o2017energy}. \revisedblue{The SBP operators form building blocks for
constructing our provably stable staggered grid scheme. To make it easier to understand how the
construction works, we share some of our codes in the repository
\url{github.com/ooreilly/sbp}. In particular, this repository contains the SBP
operators derived in this work and a prototype implementation of the
scheme.}

We introduce the grid vectors $\xb$ (node-centered) and $\hat{\xb}$
(cell-centered) that discretize the unit
interval using $N+1$ and $N+2$ grid points, respectively,
\[
\mathbf{x} = [x_0 \  x_1 \ \ldots \ x_N]^T
\quad
\mbox{and}
\quad
\hat{\mathbf{x}} =
[\hat{x}_0 \ \hat{x}_1 \ \ldots \ \hat{x}_{N+1}]^T.
\]
The grid points are
\begin{align}
        x_0 &= 0, 
        &
        x_i &= ih, 
        &
        \quad
        i&=1,2,\ldots,N-1, 
        \quad
        &
        x_N &= 1,
        \\
        \hat{x}_0 &= 0,
        &
        \hat{x}_i &= (i - 1/2)h, 
        &
        \quad
        i&=1,2,\ldots,N,
        \quad
        &
        \hat{x}_{N+1} &= 1,\label{eq_1d-grids} 
\end{align}
and $h = 1/N$ is the grid spacing, \revisedgreen{see Figure \ref{fig:grid1d}}.
On the $\xb$-grid, the \revisedgreen{function}
$u(x)$ is approximated by \revisedgreen{grid function} $\mathbf{u} = [u_0 \ u_1 \ \ldots \ u_N]^T$, where
$u_i \approx u(x_i)$. Similarly, on the $\hat{\xb}$-grid, the
\revisedgreen{function} ${v}(x)$ is
approximated by \revisedgreen{the grid function} $\hat{\mathbf{v}} = [\hat{v}_0 \ \hat{v}_1 \ \ldots
\ \hat{v}_{N+1}]^T$, where $\hat{v}_i \approx v(\hat{x}_i)$. 

We introduce the SBP staggered grid difference operators $D$ and $\hat{D}$ that approximate the first
derivative, see Figure \ref{fig:grid1d}. These operators are accurate to order
\revisedred{$2s$} in the interior and to order
\revisedred{$s$} for a few points near each boundary. The operator $D$
\revisedgreen{approximates the derivative of a grid function on the
$\hat{\mathbf{x}}$-grid at the $\mathbf{x}$-grid}.
In contrast, the operator $\hat{D}$ \revisedgreen{approximates the 
derivative of a grid function on the $\mathbf{x}$-grid at the $\hat{\mathbf{x}}$-grid}.
For example, for all difference operators that are at
least second order accurate on the boundary (\revisedred{$s \geq 2$}),
quadratic monomials \revisedgreen{$\mathbf{\xi}$} and \revisedgreen{$\hat{\xib}$} are differentiated exactly,
\[
        \hat{D}\revisedgreen{\xib} = 2\revisedgreen{\hat{\xib}}, \quad
        {D}\revisedgreen{\hat{\xib}} = 2\revisedgreen{\xib},
        \quad
\revisedgreen{\xib} = [x_0^2 \ x_1^2 \ \ldots \ x_N^2]^T, \quad
\revisedgreen{\hat{\xib}} = [\hat{x}_0^2 \ \hat{x}_1^2 \ \ldots \ \hat{x}_{N+1}^2]^T.
\]

\begin{figure}[!htbp]
        \includegraphics[width=\textwidth]{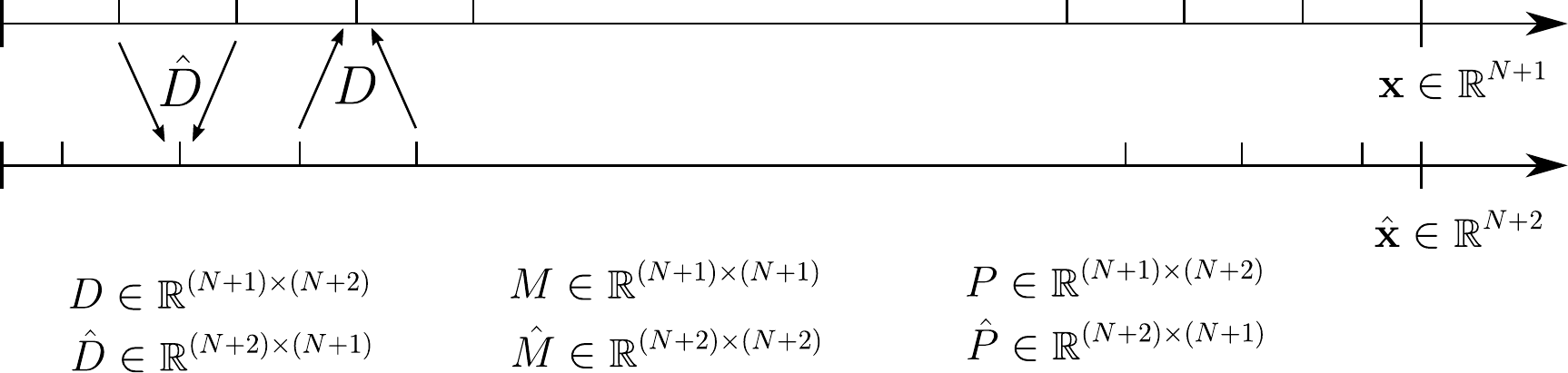}
        \caption{Action of staggered grid operators on the two grids $\xb$ and
        $\hat{\xb}$ and their respective sizes.}
        \label{fig:grid1d}
\end{figure}

Staggered SBP difference operators satisfy a summation
by parts property, corresponding to integration by parts,
\begin{align}
        \int_{x=0}^1 u \frac{\partial v}{\partial x} dx &= u(1)v(1) -
        u(0)v(0) - 
        \int_{x=0}^1 \frac{\partial u}{\partial x} v dx, 
        \nonumber
        \\
        &\nonumber\\
        \mathbf{u}^TMD\hat{\mathbf{v}} &= u_Nv_{N+1} - u_0v_0 -
        (\hat{D}\mathbf{u})^T\hat{M}\hat{\mathbf{v}}.
        \label{SBP_property}
\end{align}
The property (\ref{SBP_property}) is key to proving energy stability of the semi-discrete
approximation. The matrices $M$ and
$\hat{M}$ are positive definite diagonal
matrices that define \revisedred{$2s-1$} order accurate quadrature rules on the 
grids $\bf x$ and $\hat{\mathbf{x}}$, respectively. Since the grid points of the
cell-centered and nodal grids overlap, the SBP property (\ref{SBP_property})
gives boundary terms involving unknowns on the boundary. Alternate approaches
that avoid overlapping boundary points either interpolates \revisedgreen{or} extrapolates the numerical solution to the boundary
\cite{gao2019sbp}, or strongly impose the boundary condition \cite{prochnow2017treatment}.

In addition to difference operators, we also need to define SBP staggered interpolation
operators. These operators are \revisedred{$2s$} order accurate in the interior
and \revisedred{$s$} order
accurate on the boundary. The interpolation operator $P$
interpolates a grid function on the $\hat{\mathbf{x}}$-grid to the
$\mathbf{x}$-grid, and the interpolation operator $\hat{P}$
interpolates in the opposite direction. For example, if the interpolation operator is at
least second order accurate on the boundary (\revisedred{$s \geq 2$}), then 
\[
        P\hat{\xb} = \xb, 
        \quad
        \hat{P}{\xb} = \hat{\xb}.
\]
The interpolation operators satisfy the SBP property
\begin{align}
        \ub^T M P \hat{\vb}  = 
        (\hat{P}\ub)^T \hat{M} \hat{\vb}
        \Rightarrow 
        M P = \hat{P}^T \hat{M}, 
        \label{SBP_interpolation}
\end{align}
for all grid functions $\ub$ and $\vb$.
The property (\ref{SBP_interpolation}) states that when evaluating the integral $\int uv dx$ using
the SBP quadrature, the result must be the same whether $\ub$ is interpolated to
the $\hat{\xb}$-grid, or $\hat{\vb}$ is interpolated to the ${\xb}$-grid.

\subsection{Two-dimensional SBP Operators}\label{sec_2dsbp} 
In two spatial dimensions, we introduce a staggered grid that discretizes the
pressure field $p$ and the contravariant velocity components ($v^1$, $v^2$). This
grid is parameterized by the curvilinear coordinates $r^1$ and $r^2$, see Figure
\ref{fig_stag}.

On the staggered grid, pressure $p$ is located at
cell-centers and the contravariant velocity components $v^1$ and $v^2$
are located on the edges with $r^1=\mbox{const.}$ \revisedgreen{or}
$r^2=\mbox{const.}$, respectively. 
\begin{figure}
        \centering
        \includegraphics[width=0.8\textwidth]{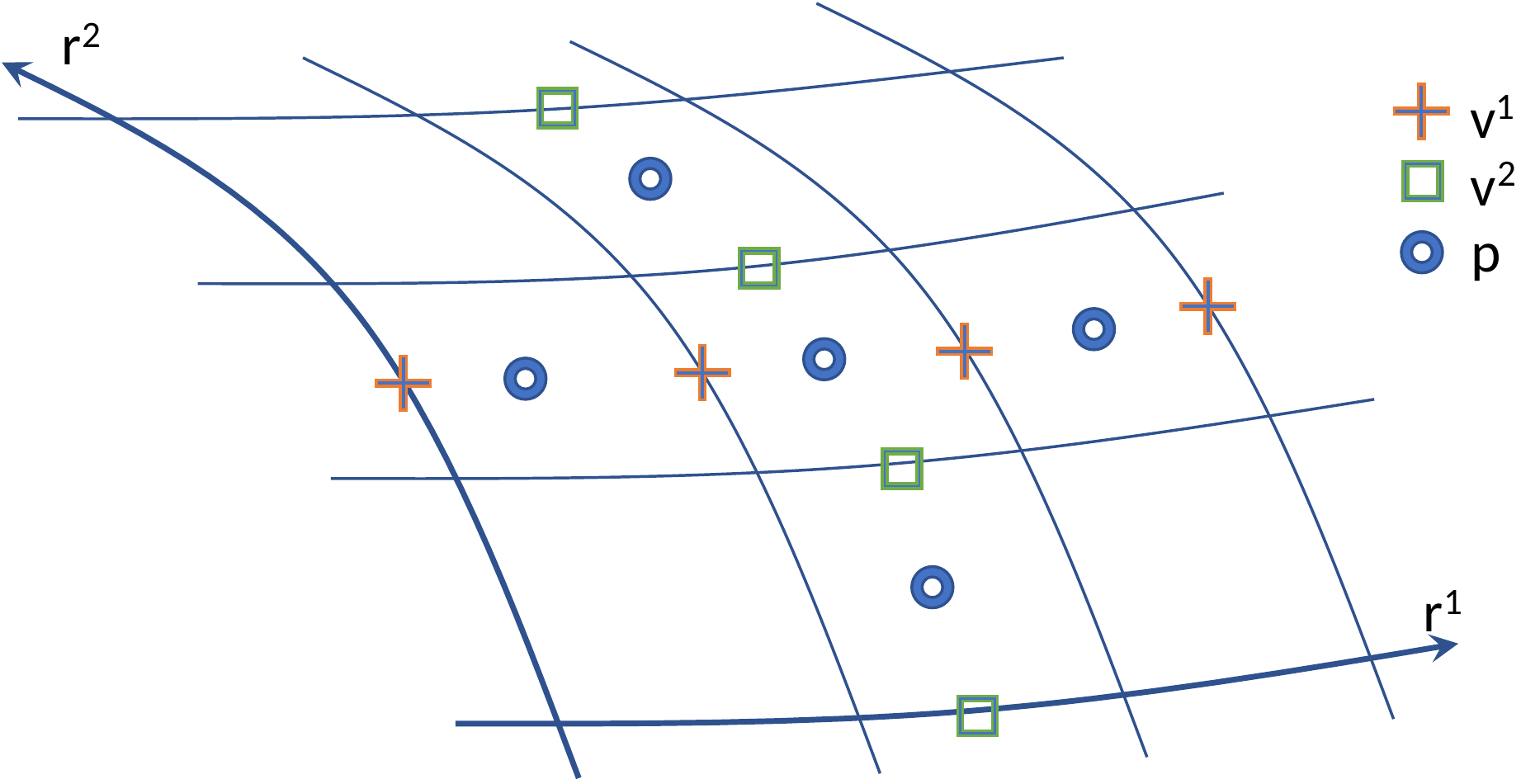}
        \caption{Curvilinear staggered grid used for discretizing the acoustic wave equation in two
          dimensions. The pressure $p$ is located at the cell centers (circles)
          and the contravariant
          components $v^i$ of the velocity field are located at the edge midpoints (pluses and
          squares).}
        \label{fig_stag}
\end{figure}
Without loss
of generality, we use the same number of cells $N$ (and grid spacing $h$) in
both
coordinate directions.
We collect the spatial discretizations of the pressure and velocity component fields in the vectors
$\mathbf{p}$, $\mathbf{v}^1$ and $\mathbf{v}^2$. 
These column vectors are ordered as follows:
\begin{align*}
        \mathbf{p}&=\left[{p}_{0,0}\ {p}_{0,1}\ \ldots\ {p}_{0,N+1}\
{p}_{1,0}\ {p}_{1,1}\ \ldots\ {p}_{1,N+1}\ \ldots
        \ \ldots\ {p}_{N+1,N+1}\right]^T \in \mathbb{R}^{\hat{N}},
\\
        \mathbf{v}^1&=\left[{v}^1_{0,0}\ {v}_{0,1}^1\ \ldots\ {v}^1_{0,N+1}\
{v}^1_{1,0}\ {v}^1_{1,1}\ \ldots\ {v}^1_{1,N+1}\ \ldots
        \ \ldots\ {v}^1_{N,N+1}\right]^T \in \mathbb{R}^{N_1},
\\
        \mathbf{v}^2 &=\left[{v}^2_{0,0}\ {v}_{0,1}^2\ \ldots\ {v}^2_{0,N}\
{v}^2_{1,0}\ {v}^2_{1,1}\ \ldots\ {v}^2_{1,N}\ \ldots
        \ \ldots\ {v}^2_{N+1,N}\right]^T \in \mathbb{R}^{N_2},
\end{align*}
where $p_{ij} \approx p(\hat{r}^1_i, \hat{r}^2_j, t)$, $v^1_{ij} \approx v^1(r^1_i,
\hat{r}^2_j, t)$, and $v^2_{ij} \approx v^2(\hat{r}^1_i, r^2_j, t)$.
Furthermore, the length
of each grid vector is, respectively,
\begin{align}
        \hat{N} = (N+2)(N+2), \quad
        N_1 = (N+1)(N+2),
        \quad
        N_2 = (N+2)(N+1),
        \nonumber
\end{align}
i.e., $N_1 = N_2$ in this particular case.

We extend the differentiation and interpolation operators to two spatial
dimensions by applying them in a line-by-line manner using Kronecker products. 
The difference operators become
\begin{align*}
        D_1 &= (D \otimes \hat{I})
        \in \mathbb{R}^{N_1 \times
        \hat{N} },
        \quad 
        D_2 = (\hat{I} \otimes D)  \in \mathbb{R}^{N_2 \times
        \hat{N}}, \\
        \widehat{D}_1 &= (\hat{D} \otimes \hat{I}) \in
        \mathbb{R}^{\hat{N} \times N_1},
        \quad
        \widehat{D}_2 = (\hat{I} \otimes \hat{D}) 
        \in
        \mathbb{R}^{\hat{N} \times N_2},
\end{align*}
where $\hat{I}$ is the $\hat{N} \times \hat{N}$ identity matrix.
In this notation, $D_1$ is a difference operator that acts on a cell-centered
quantity and approximates the first
derivative $\p/\p r^1$ on the edge-1 grid, see Figure \ref{fig_bnd}. We also introduce the
following norm weight matrices, 
\begin{align*}
        H_1 &= (M \otimes \hat{M}) \in \mathbb{R}^{N_1 \times N_1},
        \quad
        H_2 = (\hat{M} \otimes {M}) \in \mathbb{R}^{N_2 \times N_2},
        \quad
        \hat{H} = (\hat{M} \otimes \hat{M}) \in \mathbb{R}^{\hat{N} \times \hat{N}},
\end{align*}
and interpolation operators
\begin{align}
        \begin{aligned}
                P_{1c} = (P \otimes \hat{I}) \in \mathbb{R}^{N_1 \times
                \hat{N}},
                \quad
                P_{c1} = (\hat{P} \otimes \hat{I}) \in \mathbb{R}^{\hat{N}
                \times N_1},
                \\
                P_{2c} = (\hat{I} \otimes {P}) \in \mathbb{R}^{N_2 \times
                \hat{N}},
                \quad
                P_{c2} = (\hat{I} \otimes \hat{P}) \in \mathbb{R}^{\hat{N}
                \times N_2}. 
        \end{aligned}
        \label{interpolation_operators}
\end{align}
The interpolation operator $P_{1c}$ interpolates from the cell-centered grid to
the edge-1 grid, and so forth.
By applying the SBP interpolation property (\ref{SBP_interpolation}) to  (\ref{interpolation_operators}), we find
\begin{align}
 H_1P_{1c} = P_{c1}^T\hat{H}
 \quad
 \mbox{and}
 \quad
H_2P_{2c} = P_{c2}^T\hat{H}.
\label{SBP_interpolation_2d}
\end{align}

The extension of the SBP property (\ref{SBP_property}) to 2D becomes,
\begin{align}
\begin{aligned}
        (\vb^1)^TH_1 D \pb &= - (\hat{D}_1 \vb^1)^T \hat{H}\pb + (\vb^1_R)^T
        \hat{M}
        \pb_R - (\vb^1_L)^T \hat{M} \pb_L, \\
        (\vb^2)^TH_2 D \pb &= - (\hat{D}_2 \vb^2)^T \hat{H} \pb+ (\vb^2_T)^T
        \hat{M}
        \pb_T - (\vb^2_B)^T \hat{M} \pb_B, 
\end{aligned}
        \label{SBP_property_2d}
\end{align}
where $\pb_R,\ \pb_L,\ \pb_T, \ \pb_B  \in \mathbb{R}^{N+2}$, 
contain the grid values of $\pb$ along the right, left, top, and bottom
  boundaries, respectively (and similarly for $\vb^1$ and $\vb^2$).
\begin{figure}
        \centering
        \includegraphics[width=0.8\textwidth]{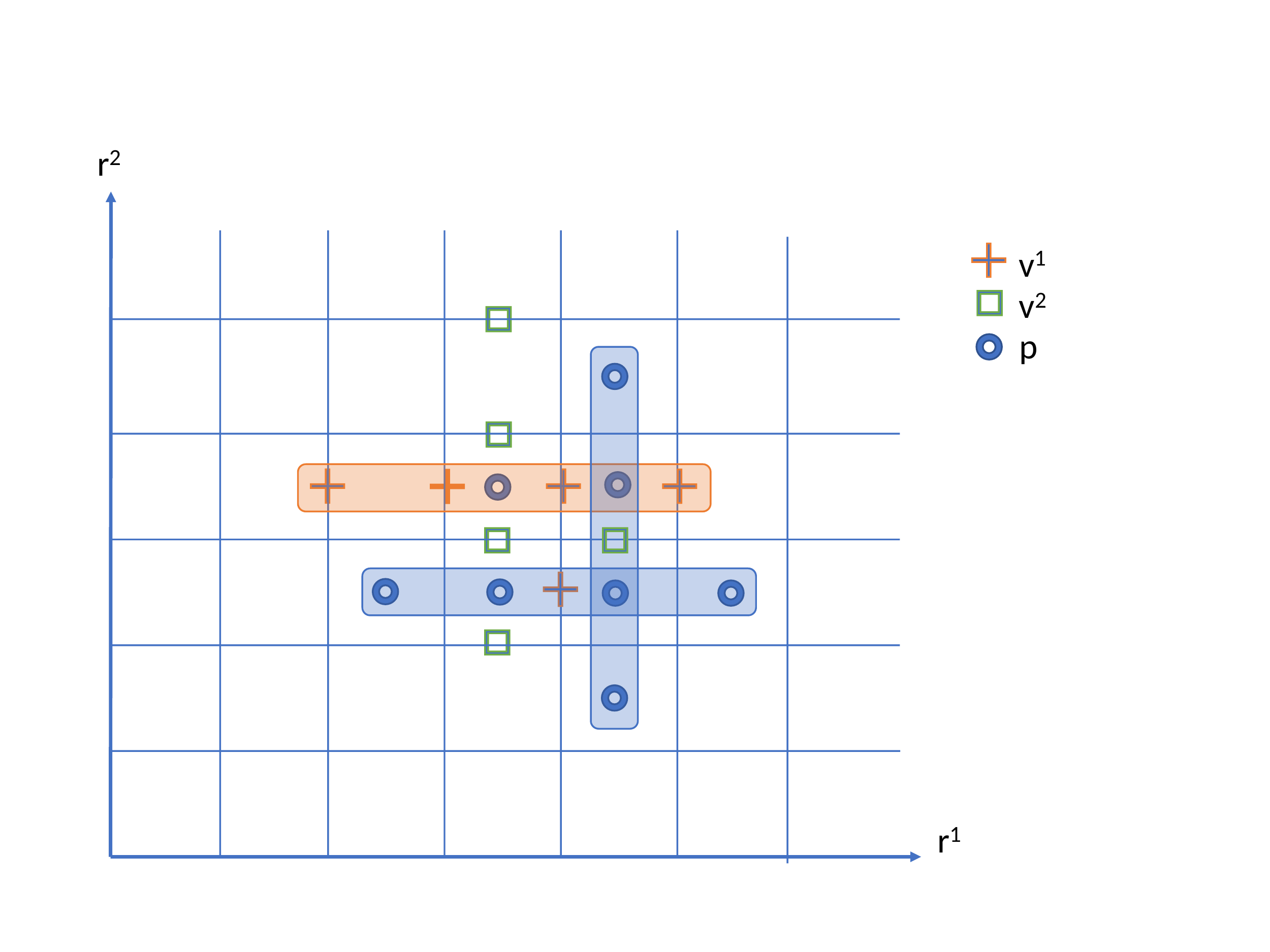}
        \caption{The interior stencils for the fourth order accurate staggered difference and interpolation operators.}
        \label{fig_bnd}
\end{figure}

\subsection{The semi-discrete approximation}
After discretizing the spatial derivatives on the staggered grid, the
semi-discretization of the pressure
equation (\ref{eq:pressure}) becomes
\begin{equation}
\hat{J} \frac{d \hat{\pb}}{d t} + \widehat{D}_1J_1{\vb}^1 + \widehat{D}_2 J_2
{\vb}^2 = 0.
\label{eq:pressure_discrete}
\end{equation}
In (\ref{eq:pressure_discrete}), $\hat{J}$, $J_1$, and $J_2$ are diagonal
matrices that hold the Jacobian evaluated at the cell-centers and the edge grid points, respectively.

The semi-discrete approximation of the equations for the contravariant velocity components (\ref{eq:velocity})
become
\begin{align}
        \frac{d}{dt}
        \begin{bmatrix} 
                \vb^1 \\
                \vb^2
        \end{bmatrix}
+ G \begin{bmatrix}
                D_1 \pb \\
                D_2 \pb
            \end{bmatrix}
            = 0.
\label{eq:velocity_discrete}
\end{align}
In (\ref{eq:velocity_discrete}), $G$ is the discrete approximation of the
contravariant metric tensor $g^{ij}$. The explicit form of $G$ will be presented in the next
section. Recall that the contravariant metric tensor gives the
transformation from covariant to contravariant components. The discrete
counterpart of (\ref{eq:contravariant}) is
\begin{align}
        \begin{bmatrix} 
                \vb^1 \\
                \vb^2
        \end{bmatrix}
        =
        G
        \begin{bmatrix} 
                \vb_1 \\
                \vb_2
        \end{bmatrix}.
        \label{contravariant_transform}
\end{align}
Since the contravariant components $\vb^1$ and $\vb^2$ are associated
with different locations of the staggered grid, $G$ must be constructed using
interpolation, 
unless the curvilinear grid is orthogonal.

\section{The energy method}\label{s:energy}
In this section, we apply the energy method to show that the scheme
(\ref{eq:pressure_discrete})-(\ref{eq:velocity_discrete}) is stable. To simplify the
presentation, we consider periodic boundary conditions and refer to Appendix
\ref{a:weak} for a
stability proof with weakly imposed boundary conditions using the SAT penalty
method.
                          
In the following, the matrix $G^{-1}$ is obtained by inverting
$G$ in (\ref{eq:velocity_discrete}). The \revisedgreen{matrix} $G^{-1}$ approximates the covariant
metric tensor $g_{ij}$. In practice, we never have to explicitly compute the
inverse. 
Consider the following discrete approximation of the energy
(\ref{continuous_energy}), $\mathcal{E} = \mathcal{E}_a + \mathcal{E}_k$, 
\begin{align}
        \mathcal{E}_a(t) = 
        \frac{1}{2} \pb^T \hat{H}\hat{J} \pb
        \quad
        \mbox{and}
        \quad
        \mathcal{E}_k(t) = 
        \frac{1}{2}
                                 \begin{bmatrix}
                                        \vb^1 \\
                                        \vb^2
                                  \end{bmatrix}^T
                                  HJ
                                  G^{-1}
                                 \begin{bmatrix}
                                        \vb^1 \\
                                        \vb^2
                                  \end{bmatrix}.
        \label{energy}
\end{align}
In (\ref{energy}), we have defined 
\begin{align}
                H = 
                \begin{bmatrix}
                        H_1 & 0\\
                        0 & H_2 
                \end{bmatrix}
                \quad
                \mbox{and}
                \quad
                J = 
                \begin{bmatrix}
                        J_1 & 0 \\
                        0 & J_2 
                \end{bmatrix}.
\end{align}
Recall that the first term on the right-hand side of (\ref{continuous_energy}) is the acoustic
energy and the two remaining terms represent the kinetic energy.  The discrete approximation of the
acoustic energy, $\mathcal{E}_a$, defines a norm of $\pb$ since the matrices $\hat{H}$ and $\hat{J}$
are diagonal with positive entries. In order for the kinetic energy, $\mathcal{E}_k$, to be a norm,
the matrix $HJG^{-1}$ must be positive definite. Since $H_1$, $H_2$, $J_1$, and $J_2$ are positive
diagonal matrices, $HJ$ is a positive diagonal matrix.  The construction of $G$ such that $HJG^{-1}$
is symmetric and positive definite is discussed in the next section.  Given such a $G$, we obtain an
energy stable scheme, as is shown in the following lemma.
\begin{lemma}\label{l:energyrate}
        Assume that the grid \revisedgreen{metric and solution} is periodic in each direction. Then the energy (\ref{energy}) of the semi-discrete
        approximation (\ref{eq:pressure_discrete})-(\ref{eq:velocity_discrete})
        is conserved.
\end{lemma}
\begin{proof}
Differentiating the acoustic energy in (\ref{energy}) with respect to time and
inserting (\ref{eq:pressure_discrete}) into the result yields
\begin{equation}
        \pb^T\widehat{H}\widehat{J} \frac{d \pb}{d t} = -
        \pb^T\widehat{H}\widehat{D}_1J_1{\vb}^1 - \pb^T\widehat{H}\widehat{D}_2 J_2 \vb^2.
        \label{acoustic_rate}
\end{equation}
Similarly, differentiating the kinetic energy in (\ref{energy}) with
        respect to time and inserting (\ref{eq:velocity_discrete}), and assuming
        $HJG^{-1} = (HJG^{-1})^T$, 
        yields
\begin{equation}
        \begin{bmatrix} 
                d\vb^1/dt \\
                d\vb^2/dt
        \end{bmatrix}^T
        HJG^{-1} 
        \begin{bmatrix} 
                \vb^1 \\
                \vb^2
        \end{bmatrix}
        = -
\begin{bmatrix}
               H_1J_1D_1 \pb \\
               H_2J_2D_2 \pb
\end{bmatrix}^T
        \begin{bmatrix} 
                \vb^1 \\
                \vb^2
        \end{bmatrix}.
        \label{kinetic_rate}
\end{equation}
The energy rate is the sum of (\ref{acoustic_rate}) and (\ref{kinetic_rate}),
\begin{align}
        \frac{d\mathcal{E}}{dt}
        = 
        - \pb^T(\widehat{H}\widehat{D}_1 + D_1^TH_1)J_1\vb^1 
        - \pb^T(\widehat{H}\widehat{D}_2 + D_2^TH_2)J_2\vb^2.
        \label{energy_rate}
\end{align}
        On a \revisedgreen{two dimensional} periodic grid, \revisedgreen{the boundary terms in
        (\ref{SBP_property_2d}) cancel out and therefore the SBP property} simplifies to
$\widehat{H}\widehat{D}_1 = - D_1^TH_1$
and $\widehat{H}\widehat{D}_2 = -D_2^TH_2$. Thus, the right-hand side of
(\ref{energy_rate}) is identically zero and the energy of the semi-discrete
approximation is conserved.
\end{proof}

\section{Discretization of the contravariant metric tensor}\label{s:contravariant}
To define a norm for $(\vb^1, \vb^2)$, note that the second matrix in
(\ref{energy}) can be written
as
\[
        HJG^{-1} = HJ(HJG)^{-1}HJ.
\]
Because $HJ$ is diagonal and positive definite, it is sufficient to show that
$HJG$ is symmetric and positive definite.

There are many ways to construct the matrix $G$ to form the approximation $\mathcal{E}_k$ of
the kinetic energy in (\ref{energy}), but care is needed to preserve energy positivity. What complicates
matters is the fact that $\vb^1$ and $\vb^2$ are defined on
different grids. Therefore, the construction of $G$ must involve interpolation
operators at least in its off-diagonal blocks. However, as we will see in
Section \ref{s:conditional_metric_tensor}, \revisedgreen{
 positive definiteness may be lost when only 
the off-diagonal blocks are interpolated}.

\subsection{Unconditionally energy positivity preserving construction}
To obtain a matrix $G$ that gives a positive definite $HJG$, we evaluate the contravariant
metric tensor and Jacobian at the cell-centered grid points, and then interpolate back and
forth to each velocity grid. This approach results in
\begin{align}
        G =
        \begin{bmatrix}
                J_1^{-1}P_{1c}\hat{J}\hat{g}^{11} P_{c1}  &
                J_1^{-1}P_{1c}\hat{J}\hat{g}^{12}P_{c2} \\
                J_2^{-1}P_{2c}\hat{J}\hat{g}^{12} P_{c1}  &
                J_2^{-1}P_{2c}\hat{J}\hat{g}^{22}P_{c2}
        \end{bmatrix}.
        \label{stable_G}
\end{align}
The reason why $J$ appears in the construction of $G$ will become clear next.
To show the following result, we first need to make an assumption.
\begin{assumption}\label{a:null}
        The null-space of the SBP interpolation operators
        in (\ref{interpolation_operators}) is empty.
\end{assumption}
We have verified this assumption through numerical experiments. We are now ready to state the following result.
\begin{lemma}\label{l:pos}
        Let $HJ$ be given by (\ref{energy}) and let $G$ be defined by (\ref{stable_G}). The matrix
        $HJG$ is positive definite if Assumption \ref{a:null} holds, $\hat{J} > 0$ and $\hat{g}^{ij}$ is
        positive definite.
\end{lemma}
\begin{proof}
Multiplying (\ref{stable_G}) by
$HJ$ results in
\[
        HJG =
        \begin{bmatrix}
                H_1P_{1c}\hat{J}\hat{g}^{11} P_{c1}  &
                H_1P_{1c}\hat{J}\hat{g}^{12}P_{c2} \\
                H_2P_{2c}\hat{J}\hat{g}^{12} P_{c1}  &
                H_2P_{2c}\hat{J}\hat{g}^{22}P_{c2}
        \end{bmatrix}.
\]
Using the SBP interpolation property (\ref{SBP_interpolation_2d}), 
\begin{align}
        HJG =
        \begin{bmatrix}
                P_{1c}^T\hat{H}\hat{J}\hat{g}^{11}P_{c1}  &
                P_{1c}^T\hat{H}\hat{J}\hat{g}^{12}P_{c2} \\
                P_{2c}^T\hat{H}\hat{J}\hat{g}^{12}P_{c1}  &
                P_{2c}^T\hat{H}\hat{J}\hat{g}^{22}P_{c2}
        \end{bmatrix}.
        \nonumber
\end{align}
By factoring out the interpolation and norm operators as well as the Jacobians, we find
\begin{align}
        HJG =
        \begin{bmatrix}
                P_{c1}^T & 0 \\
                0 & P_{c2}^T
        \end{bmatrix}
        \begin{bmatrix}
                \hat{H}\hat{J}  &
                0 \\
                0  &
                \hat{H}\hat{J}
        \end{bmatrix}^{1/2}
        \underbrace{
        \begin{bmatrix}
                \hat{g}^{11}  &
                \hat{g}^{12} \\
                \hat{g}^{12}  &
                \hat{g}^{22}
        \end{bmatrix}
}_{V}
        \begin{bmatrix}
                \hat{H}\hat{J}  &
                0 \\
                0  &
                \hat{H}\hat{J}
        \end{bmatrix}^{1/2}
        \begin{bmatrix}
        P_{c1} & 0\\
         0 & P_{c2}
\end{bmatrix}.
\label{HJG}
\end{align}
        This factorization relies on the fact that the matrices $\hat{J}$, $\hat{H}$,
        $\hat{g}^{11}$,  $\hat{g}^{12}$, and $\hat{g}^{22}$ are diagonal with
        positive entries, and
        hence commute. Next, we show that the matrix $V$ is positive definite. 
        We can permute
        $V$ so that it consists of positive definite $2 \times 2$ blocks, one per grid point. Let $R$
        be a permutation matrix that is obtained by permuting the rows of the
        identity matrix. Then $\tilde{V} = R^T V R$ is the permutation of $V$
        that consists of $2 \times 2$ positive definite blocks. Each block holds
        the contravariant metric tensor corresponding to one grid point, which
        is non singular by assumption.
        Since permutation
        matrices satisfy $RR^T = I$, we have that $V = RR^T V RR^T =
        R\tilde{V}R^T$ is positive definite.
        If Assumption \ref{a:null} holds, then $HJG$ is symmetric and positive
        definite.
\end{proof}

\subsection{Conditionally energy positivity preserving
construction}\label{s:conditional_metric_tensor}
While the particular choice of $G$ in (\ref{stable_G}) yields a positive kinetic
energy for all non singular mappings, it may be undesirable to use this choice in
practice. The reason is that blocks along the diagonal of $G$ contain interpolation
operators that interpolate back and forth between two different grids. 
This interpolation adds extra computation and may also reduce
the accuracy in the numerical solution. Next, we
show that it is possible to avoid these extra computations and improve the
accuracy. However, care has to be taken to ensure that the 
contravariant metric tensor is discretized such that it remains positive definite. 

Consider the following discretization,
\begin{align}
        \widetilde{G} =
        \begin{bmatrix}
                g^{11}_1  &
                J_1^{-1}P_{1c}\hat{J}\hat{g}^{12}P_{c2} \\
                J_2^{-1}P_{2c}\hat{J}\hat{g}^{12}P_{c1}  &
                g^{22}_2
        \end{bmatrix}.
        \label{G_mod}
\end{align}
This choice preserves the symmetry property $HJ\widetilde{G} =
(HJ\widetilde{G})^T$, but may cause a loss of positive definiteness. To avoid
having to perform expensive eigenvalue computations to assert whether $HJ\widetilde{G}$ is
positive definite, we derive and analyze an approximate condition that can be
evaluated in a point-wise manner. 

Introduce the coefficients $\alpha > 0$ and $\beta > 0$. We use the SBP interpolation relations
(\ref{SBP_interpolation_2d}) to decompose the matrix (\ref{G_mod}) according to
\[
 HJ\widetilde{G} = 
        \begin{bmatrix}
                H_1 J_1 g^{11}_1  &
                H_1P_{1c}\hat{J}\hat{g}^{12}P_{c2} \\
                H_2P_{2c}\hat{J}\hat{g}^{12}P_{c1}  &
                H_2J_2g^{22}_2
        \end{bmatrix}
        =:
 A(\alpha,\beta) + B(\alpha,\beta),
\]
where 
\begin{align}
        A(\alpha,\beta) =
        \begin{bmatrix}
                P_{c1}^T & 0 \\
                0 & P_{c2}^T
        \end{bmatrix}
        \underbrace{
        \begin{bmatrix}
                \alpha \hat{H}\hat{J}\hat{g}^{11} &  \hat{H}\hat{J}\hat{g}^{12} \\
                 \hat{H}\hat{J}\hat{g}^{12} & \beta  \hat{H}\hat{J}\hat{g}^{22}
        \end{bmatrix}}_{C(\alpha,\beta)}
        \begin{bmatrix}
                P_{c1} & 0  \\
                0 & P_{c2}
        \end{bmatrix}, 
\label{Cmatrix}
\end{align}
and 
\begin{align}
B(\alpha,\beta) &=
        \begin{bmatrix}
          B_{11}(\alpha) & 0 \\
          0 & B_{22}(\beta)
        \end{bmatrix},
        \quad \left\{\begin{array}{rl}
          B_{11}(\alpha) &\!\!\!= H_1J_1g^{11}_1 - \alpha P_{c1}^T\hat{H}\hat{J}\hat{g}^{11}P_{c1},\\
          B_{22}(\beta) &\!\!\!= H_2J_2g^{22}_2 - \beta P_{c2}^T\hat{H}\hat{J}\hat{g}^{22}P_{c2}.
        \end{array}\right.
        \label{dG}
\end{align}
We can guarantee that $HJ\widetilde{G}$ is positive definite if the matrix $A(\alpha,\beta)$ is
positive definite and $B(\alpha,\beta)$ is positive semi-definite. The matrix $A(\alpha,\beta)$
becomes indefinite if $\alpha<=0$ or $\beta<=0$. We must therefore assume that the coefficients
$\alpha$ and $\beta$ are positive. On the other hand, the matrix $B(\alpha,\beta)$ becomes indefinite if
$\alpha$ or $\beta$ are too large. Our strategy is therefore to first derive upper bounds on
$\alpha$ and $\beta$ that guarantee that $B(\alpha,\beta)$ is positive
semi-definite. The definiteness
of $A(\alpha,\beta)$ can then be verified by solving a $2\times 2$ eigenvalue problem at each cell
center.

The matrix $B(\alpha,\beta)$ in (\ref{dG}) is clearly positive semi-definite if $B_{11}(\alpha)$ and
$B_{22}(\beta)$ are positive semi-definite, which can be guaranteed by bounding $\alpha$ and $\beta$
as is shown in the following lemma.
\begin{lemma}\label{l:bound} 
  The matrices $B_{11}(\alpha)$ and $B_{22}(\beta)$ are positive definite if the coefficients
  $\alpha$ and $\beta$ satisfy the bounds
  \begin{align}
    \alpha \leq \widetilde{\alpha} = \frac{1}{\max{\lambda^{(1)}}} \quad
    \mbox{and} \quad
    \beta \leq \widetilde{\beta} = \frac{1}{\max{\lambda^{(2)}}}.
  \end{align}
  Here, $\max \lambda^{(1)}$ and $\max \lambda^{(2)}$ are the largest eigenvalues of the respective
  generalized eigenvalue problems,
  \begin{align}
    {\mathbb Y}^{(1)}\boldsymbol{y} = \lambda^{(1)} {\mathbb X}^{(1)}\boldsymbol{y},\quad
    {\mathbb Y}^{(1)} = P_{c1}^T\hat{H}\hat{J}\hat{g}^{11}P_{c1},\quad {\mathbb X}^{(1)} =
    H_1J_1g^{11}_1,\label{eq_geneig1}\\
    {\mathbb Y}^{(2)}\boldsymbol{y} = \lambda^{(2)} {\mathbb X}^{(2)}\boldsymbol{y},\quad
    {\mathbb Y}^{(2)} = P_{c2}^T\hat{H}\hat{J}\hat{g}^{22}P_{c2},\quad {\mathbb X}^{(2)} =
    H_2J_2g^{22}_2. \label{eq_geneig2}
  \end{align}
\end{lemma}
\begin{proof}
Both bounds can be derived in the same fashion and we focus on the bound for $\alpha$. From the
definitions above,
\[
B_{11}(\alpha) = {\mathbb X}^{(1)} - \alpha {\mathbb Y}^{(1)}.
\]
The matrices ${\mathbb X}^{(1)}$ and ${\mathbb Y}^{(1)}$ are real, symmetric and positive definite. All
eigenvalues $\lambda^{(1)}_j$ of the generalized eigenvalue problem \eqref{eq_geneig1} are therefore real and
positive. Furthermore, the corresponding eigenvectors $\boldsymbol{y}_j$ form a complete set and can
be normalized such that
\[
\boldsymbol{y}_j^T {\mathbb X}^{(1)} \boldsymbol{y}_k = \begin{cases} 1,& j=k,\\ 0, & j\ne k.\end{cases}
\]
Any vector $\boldsymbol{z}$ can therefore be expanded according to
\[
\boldsymbol{z} = \sum_{j} \mu_j \boldsymbol{y}_j,\quad \mu_j = \boldsymbol{y}_j^T {\mathbb X}^{(1)}\boldsymbol{z}.
\]
Thus,
\begin{multline*}
  \boldsymbol{z}^T \left( {\mathbb X}^{(1)} - \alpha {\mathbb Y}^{(1)} \right)   \boldsymbol{z} =
  \left(\sum_{j} \mu_j \boldsymbol{y}_j^T\right) \left( {\mathbb X}^{(1)} - \alpha {\mathbb Y}^{(1)}
  \right)\left(\sum_k \mu_k \boldsymbol{y}_k\right) \\
  = \left(\sum_{j} \mu_j \boldsymbol{y}_j^T\right) \left( \sum_k \mu_k ( 1 - \alpha \lambda^{(1)}_k)
  {\mathbb X}^{(1)} \boldsymbol{y}_k \right) = \sum_j \mu_j^2 \left( 1 - \alpha \lambda^{(1)}_j \right).
\end{multline*}
The matrix $B_{11}(\alpha)$ is therefore positive semi-definite if
\[
 1 - \alpha \lambda^{(1)}_j \geq 0,\quad \Rightarrow \quad \alpha \leq \frac{1}{\lambda^{(1)}_j}.
\]
The inequality must be satisfied for all eigenvalues $\lambda^{(1)}_j$. We have
$1/\max(\lambda^{(1)}) \leq 1/\lambda^{(1)}_j \leq 1/\min(\lambda^{(1)})$,
which shows that $B_{11}(\alpha)$ is positive semi-definite as long as
\[
\alpha \leq \widetilde{\alpha} = \frac{1}{\max \lambda^{(1)}}.
\]
\end{proof}
In Appendix~\ref{app_eigenvals} we show how the above 2-D generalized eigenvalue problems decouples into a
number of regular eigenvalue problems along each 1-D grid line.

When using the modified metric tensor discretization, the norm $\|P\hat{P}\|_2$
should be close to \revisedred{one} to avoid instability. During the construction of the
SBP interpolation operators, there can be undetermined coefficients remaining after
satisfying the SBP property (\ref{SBP_interpolation}) and accuracy conditions.
These coefficients can be determined by choosing an objective function that 
minimizes $\|P\hat{P}\|_2$.  In this work we optimize for accuracy and manually
check that $\|P\hat{P}\|_2 \approx 1$. Please see
\cite{o2017energy,prochnow2017treatment} for further construction details. 

To understand why $\|P\hat{P}\|_2 \approx 1$ is desirable, assume the mapping is
linear. Then the metrics $J$ and
$g^{ij}$ are constant, and the generalized
eigenvalue problem (\ref{eq_geneig1}) simplifies to
\[
        P_{c1}^T \hat{H} P_{c1} \yb  = \lambda^{(1)} H_1\yb.
\]
\revisedgreen{Thereafter}, by applying the SBP property
(\ref{SBP_interpolation_2d}), we get
\[
        P_{1c}P_{c1} \yb  = \lambda^{(1)}\yb,
\]
By observing that  $P_{1c} = (P \otimes
\hat{I})$ and $P_{c1} = (\hat{P} \otimes \hat{I})$, we obtain an eigenvalue
problem along a single grid line in the $r^1$-direction, $P\hat{P} \xb  =
\lambda^{(1)}\xb$. The maximum eigenvalue $\max(\lambda^{(1)})$ satisfies $1 \leq
\max(\lambda^{(1)}) \leq \|P\hat{P}\|_2$. The lower bound comes from observing
that the constant \revisedgreen{vector} $\oneb=[1 \ 1 \ \ldots \ 1]^T$ is interpolated exactly, and is therefore an eigenvector of $P\hat{P}$.

To investigate how the size of $\|P\hat{P}\|_2$ influences the stability of the
scheme, we construct one pair of interpolation operators
with $\|P\hat{P}\|_2 =1.04$ and another pair with $\|P\hat{P}\|_2 = 8.53$. 
For the
test, we consider an almost square domain where three sides are straight and the
fourth side is a Gaussian hill,
\begin{align}
        x(r^1,r^2) = r^1, \quad
        \revisedgreen{y(r^1,r^2) = r^2 \left( 1 +  \gamma e^{-50(r^1-1/2)^2}
        \right)}. 
        \label{gaussian}
\end{align}
We increase the amplitude $\gamma$ of the Gaussian hill and monitor the minimum eigenvalue
of the symmetric matrix $HJ\widetilde{G}$ in the kinetic energy. This matrix must be
positive definite to guarantee stability. We expect that for a sufficiently large amplitude of the
Gaussian hill the minimum eigenvalue of $HJ\widetilde{G}$ becomes negative. In the test, the grids
are coarse with $17 \times 17$ grid points. Figure \ref{fig:grid} show the grids at which
$HJ\widetilde{G}$ become positive semi-definite for each set of interpolation operators. As can been
seen, if $\|P\hat{P}\|_2 \approx 1$, the grid can undergo significant deformation before
$HJ\widetilde{G}$ becomes positive semi-definite (Figure \ref{fig:grid}a). On the other hand, if
$\|P\hat{P}\|_2 \gg 1$, the grid can only undergo small deformations before $HJ\widetilde{G}$
becomes positive semi-definite (Figure \ref{fig:grid}b).

\begin{figure}[!htbp]
        \centering
 \begin{minipage}{.45\textwidth}
\centering
\subfloat[$\|P\hat{P}\|_2=1.04$ ]
{\includegraphics[scale=0.39]{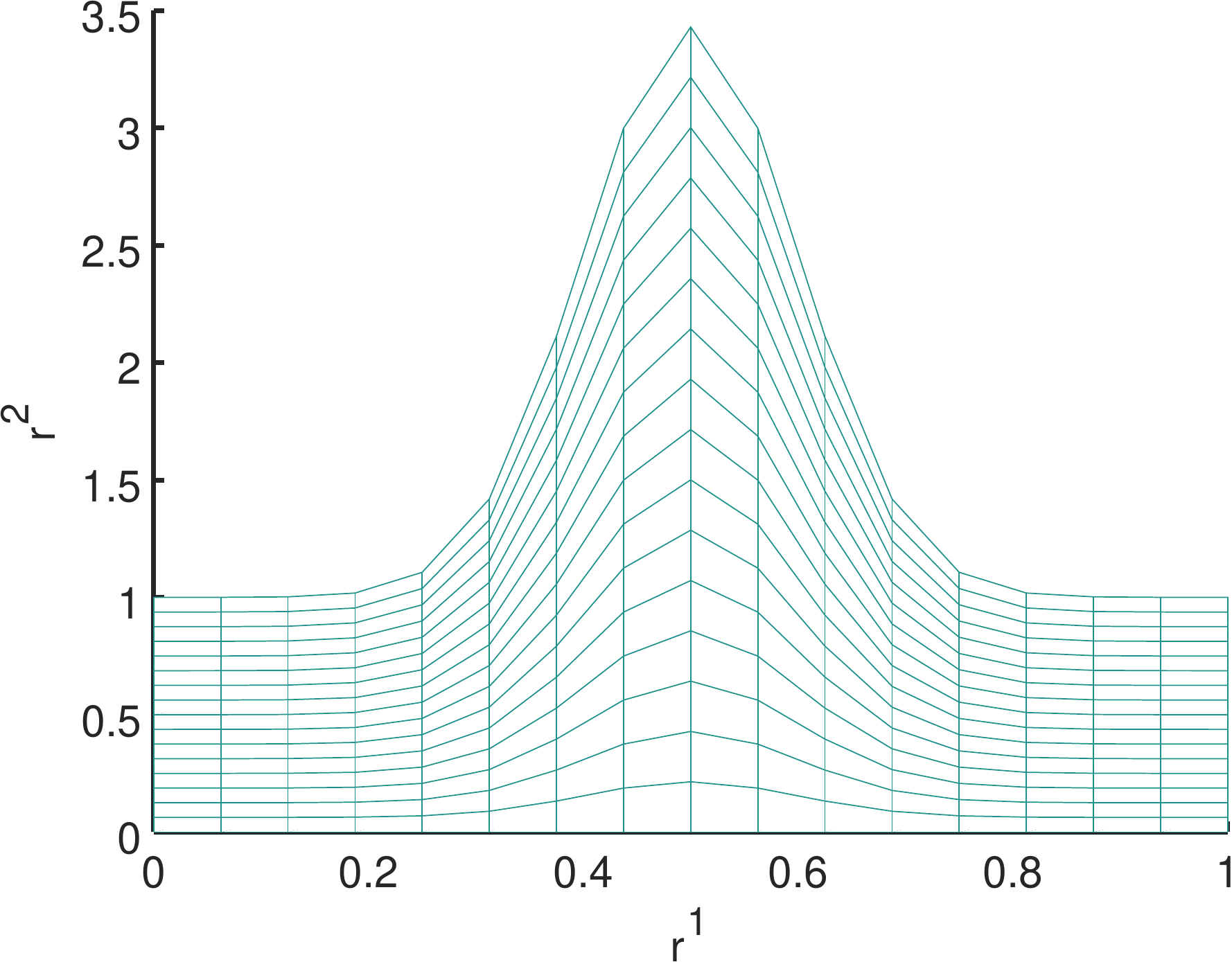}} 
\end{minipage}\qquad
\begin{minipage}{.45\textwidth}
\centering
        \subfloat[$\|P\hat{P}\|_2=8.53$ ]
        {\includegraphics[scale=0.39]{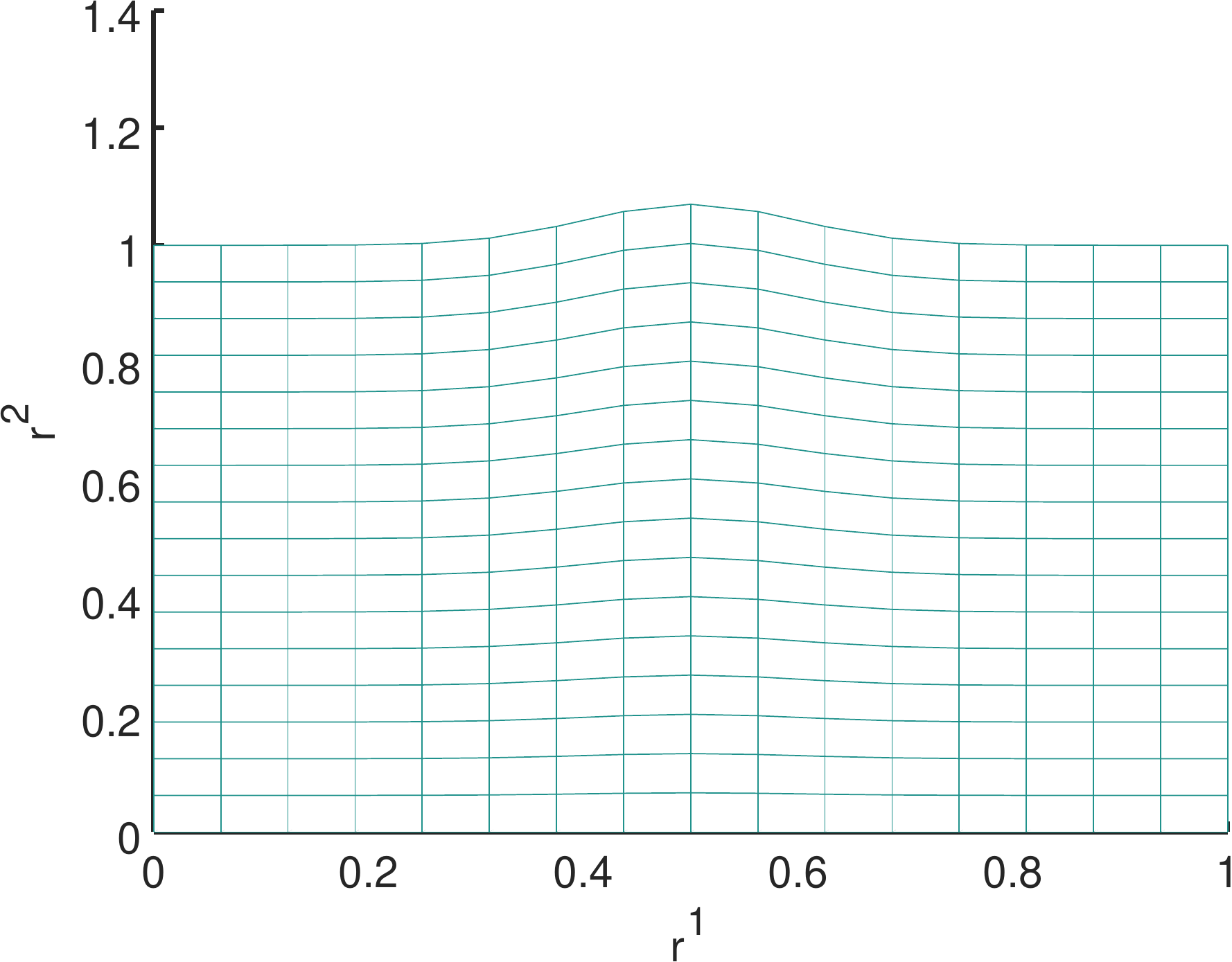}}
\end{minipage}
        \caption{ Curvilinear grids generated by the Gaussian hill mapping function
          (\ref{gaussian}). We show the largest amplitude ($\gamma$) of the Gaussian hill for which
          the kinetic energy matrix $HJ\tilde{G}$ is positive definite using two different
          interpolation operators.
        }
        \label{fig:grid}
\end{figure}

Since it is computationally expensive to compute the minimum eigenvalue $\min(\lambda)$ of
$HJ\tilde{G}$, we demonstrate how to estimate $\min(\lambda)$ using Lemma \ref{l:bound}. Note that
the constants $\alpha$ and $\beta$ can be bounded using Lemma \ref{l:bound} without much
computational work because these bounds only involve solving 1-D eigenvalue problems. Once $\alpha$
and $\beta$ are bounded, the matrix $C(\alpha,\beta)$ in \eqref{Cmatrix} can be evaluated in a
point-wise manner. This approach provides an efficient way of estimating whether $HJ\widetilde{G}$ is
positive definite. If $C(\alpha,\beta)$ is positive definite at all grid points, then
$HJ\widetilde{G}$ is positive definite. However, if $C(\alpha,\beta)$ has a negative eigenvalue at
any grid point, then the test for definiteness of $HJ\widetilde{G}$ is inconclusive.

Figure \ref{fig:stability} shows how accurate the approximate method is at
estimating if $HJ\widetilde{G}$ is positive definite by computing the minimum
eigenvalue of $C(\alpha,\beta)$ for all grid points.
In Figure \ref{fig:stability}a, the approximate method estimates that $HJ\widetilde{G}$ is
positive definite for 
$\gamma < 0.5$. (the actual amplitude at which $HJ\widetilde{G}$ becomes
positive semi-definite is $\gamma \approx 2.5$).
In Figure \ref{fig:stability}b, 
the approximate method estimates that 
$HJ\widetilde{G}$ is positive definite for amplitudes $\gamma < 0.02$ (the actual
amplitude at which $HJ\widetilde{G}$ becomes positive semi-definite is
$\gamma \approx 0.07$). 

\begin{figure}[!htbp]
        \centering
 \begin{minipage}{.4\textwidth}
\centering
\subfloat[$\|P\hat{P}\|_2=1.04$ ]
{\includegraphics[scale=0.55]{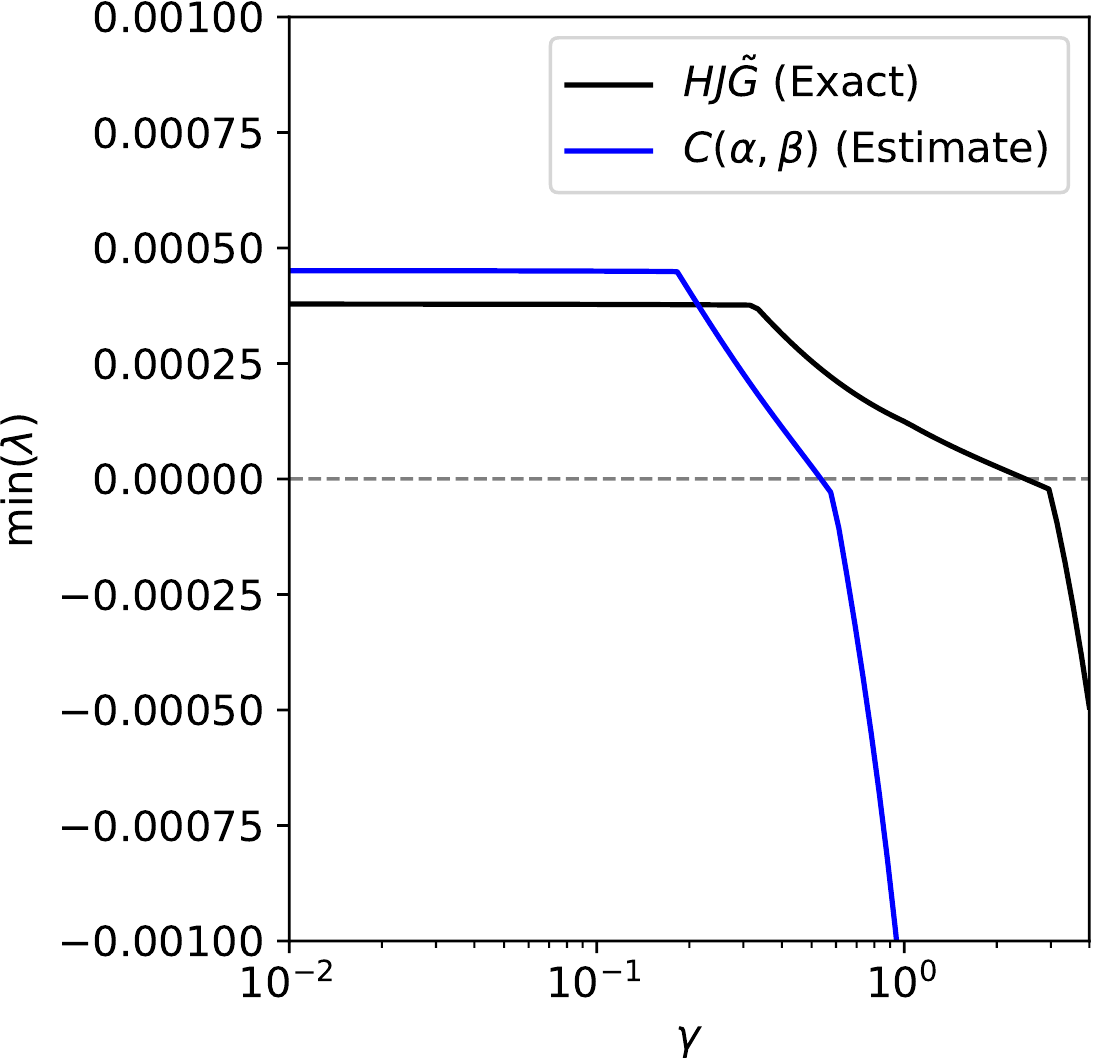}} 
\end{minipage}\qquad
\begin{minipage}{.4\textwidth}
\centering
        \subfloat[$\|P\hat{P}\|_2=8.53$ ]
        {\includegraphics[scale=0.55]{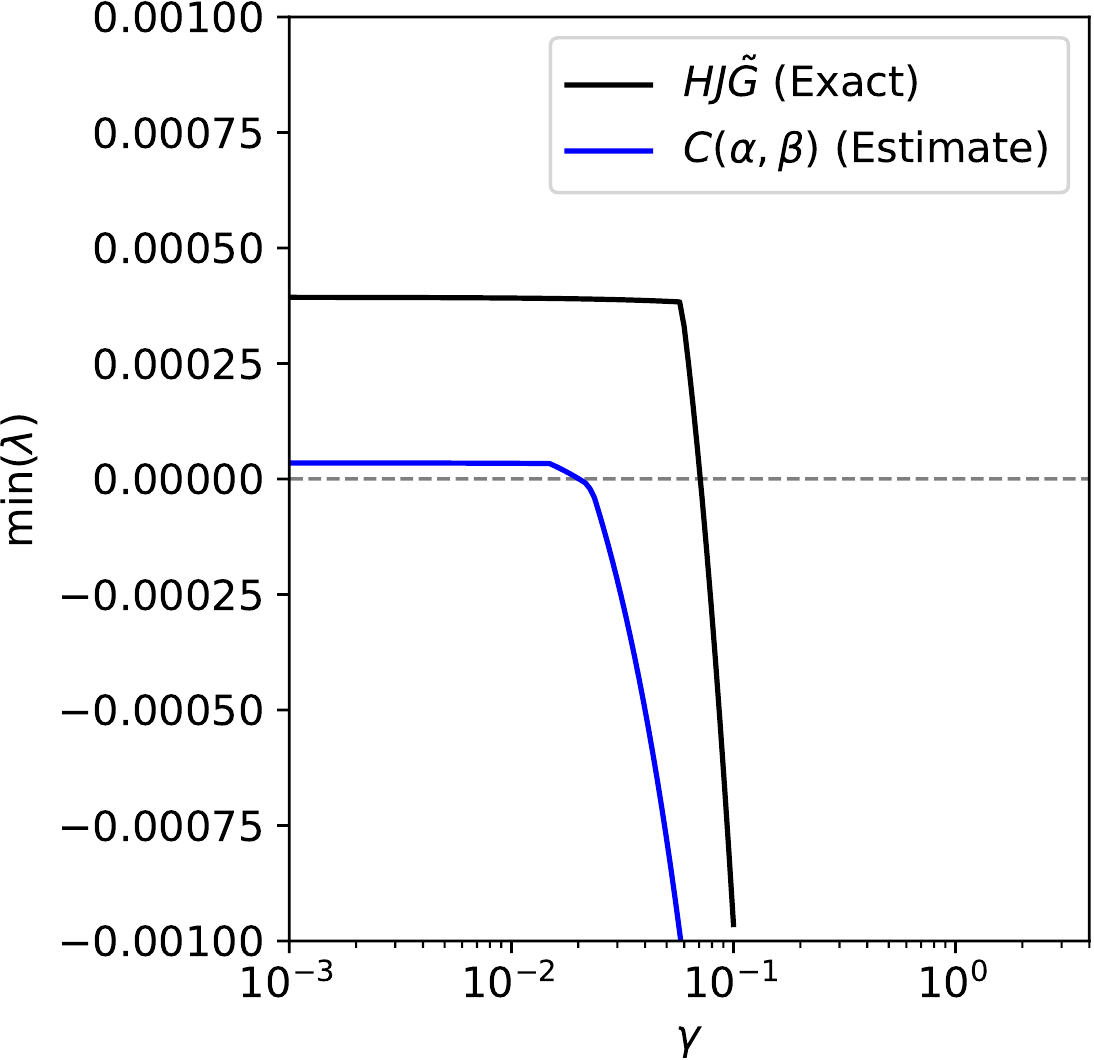}}
\end{minipage}
        \caption{ The smallest eigenvalue of the kinetic energy matrix $HJ\tilde{G}$ as function of
          the amplitude ($\gamma$) of the Gaussian hill for two different interpolation
          operators. The black lines indicate the smallest eigenvalue of $HJ\tilde{G}$, while the
          blue lines show the smallest eigenvalue of the matrix $C(\alpha,\beta)$ using the Lemma
          \ref{l:bound}.  Note that interpolation operators with a small norm $\|P\hat{P}\|_2$ make
          the scheme stable for significantly larger grid deformations compared to interpolation
          operators with a large norm.  }
        \label{fig:stability}
\end{figure}

In conclusion, large values of $\|P\hat{P}\|_2$ can
severely limit the amount of grid skewness that the scheme can tolerate without
becoming unstable. This restriction may be so severe that stable
computations on realistic grids become impossible. Therefore, if one wants to use the approximation (\ref{G_mod}) of
the contravariant metric tensor, it is important to construct interpolation
operators such that $\|P\hat{P}\|_2 \approx 1$. 

\section{Numerical experiments}\label{s:experiments}
We use the method of manufactured solutions to conduct convergence rate studies
\cite{roache2002code}. The following manufactured solution is designed to
satisfy the acoustic wave equation \eqref{vec_pressure}-\eqref{vec_velocity} in
Cartesian velocity components ($\boldsymbol{v} = \tilde{v}_x \boldsymbol{e}_x +
\tilde{v}_y \boldsymbol{e}_y$),
\begin{align}
\begin{aligned}
        \tilde{p}(x,y,t)   &= \sin(2\pi x)\sin(2\pi y)\cos(2\sqrt{2}\pi t), \\
        \tilde{v}_x(x,y,t) &= -\frac{1}{\sqrt{2}}\cos(2\pi x)\sin(2\pi y)\sin(2\sqrt{2}\pi t),\\
        \tilde{v}_y(x,y,t) &= -\frac{1}{\sqrt{2}}\sin(2\pi x)\cos(2\pi y)\sin(2\sqrt{2}\pi t).
\end{aligned}
        \label{mms}
\end{align}
To apply the manufactured solution, we impose the 
inhomogeneous pressure boundary condition, \revisedgreen{$p(x,y,t) =
\tilde{p}(x,y,t)$} for all points
$(x,y)$ on the boundary. 
We use the classical fourth order Runge-Kutta
scheme to advance the solution in time. Since a naive implementation of time
dependent boundary data reduces the convergence rate to second order for
Runge-Kutta schemes, we apply the boundary correction procedure described in
\cite{Carpenter-etal-95}.

The error is the difference between the numerical and manufactured solution, and
it is measured using either the discrete $l_2$ and $l_{\infty}$ (max) norm for
each field. The $l_2$ errors are
\begin{align}
\begin{aligned}
        \| \pb   - \tilde{\pb} \|_h^2 &= h^2\sum_{ij}\left( p_{ij} -
        \tilde{p}(\hat{x}_i,\hat{y}_j,t) \right)^2\\
        \| \vb^1 - \tilde{\vb}^1 \|_h^2 &= h^2\sum_{ij}\left( (v^1)_{ij} -
        \tilde{v}^1({x}_i,\hat{y}_j,t)\right)^2 \\
        \| \vb^2 - \tilde{\vb}^2 \|_h^2 &= h^2\sum_{ij}\left( (v^2)_{ij} -
        \tilde{v}^2(\hat{x}_i,{y}_j,t)\right)^2, 
\end{aligned}
\label{mms2}
\end{align}
where we used grid function notation $p_{ij}\approx p(\hat{x}_i,\hat{y}_j,t)$
etc. The sum of errors for all fields is denoted by 
\begin{align}
        \|\mbox{Err}\|_l = 
        \|\pb   - \tilde{\pb}\|_l + 
        \|\vb^1 - \tilde{\vb}^1\|_l + 
        \|\vb^2 - \tilde{\vb}^2\|_l,
\end{align}
where $l=h, \infty$ (depending on the choice of norm). Let $q_{h}$ and $q_{\infty}$ denote the convergence rates in the
$l_2$ and $l_{\infty}$ norms, respectively.
In (\ref{mms}), the contravariant components of the manufactured solution are
obtained by the formulas
\begin{align}
        \tilde{v}^1 = \tilde{v}_x(\ab^1\cdot \eb_x) + \tilde{v}_y(\ab^1 \cdot \eb_y), \
        \tilde{v}^2 = \tilde{v}_x(\ab^2\cdot \eb_x) + \tilde{v}_y(\ab^2 \cdot \eb_y),
        \label{cartesian_to_contravariant}
\end{align}
where $\eb_x = (1,0)$ and $\eb_y = (0,1)$ are the Cartesian basis vectors. In
all numerical experiments, we use staggered SBP interpolation and
differentiation operators having fourth order interior accuracy and second order
boundary accuracy. To construct the metrics in the scheme, we use the same
staggered SBP interpolation and differentiation operators for discretizing the
covariant basis vectors $\ab_1$ and $\ab_2$, Jacobian $J$. The discretized
contravariant basis vectors $\ab^1$ and $\ab^2$ are obtained from
(\ref{contravariant_basis}). 

\subsection{Comparison of metric tensors}\label{sec_exp1} As discussed in
section \ref{s:contravariant}, the discretization of the contravariant metric
tensor $G$ in (\ref{stable_G}) results in a provably stable scheme. However,
\revisedgreen{as we proceed to show, the modified discretization $\widetilde{G}$
in (\ref{G_mod}) gives better
accuracy}.

For this test, we use linear transfinite interpolation \cite{ThoWarMas85} to define a mapping by
specifying the following parameterization of the four boundary segments
\begin{align}
\begin{aligned}
        x(0, r^2) &= x_0 - a \sin(kr^2)  & 
        y(0, r^2) &= y_0 + r^2 & (\mbox{left}) \\
        x(1, r^2) &= x_0 + 1 + a \sin(kr^2) & 
        y(1, r^2) &= y_0 + r^2 &
        (\mbox{right})\\
        x(r^1, 0) &= x_0 + r^1  & 
        y(r^1, 0) &= y_0 - a \sin(kr^1) & (\mbox{bottom}) \\
        x(r^1, 1) &= x_0 + r^1  & 
        y(r^1, 1) &= y_0 + 1 + a \sin(kr^1), & (\mbox{top}) 
        \label{test1}
\end{aligned}
\end{align}
where $a=0.05$, $k =2\pi$, $x_0 = y_0 = 0.2$. 

\begin{table}[!htbp]
\centering
\begin{tabularx}{\textwidth}{lllllllll}
\toprule
  N & 
  $\|\text{Error}^{(G)}\|_h$ & $\|\text{Error}^{(G)}\|_{\infty}$ & $q^{(G)}_h$ &
  $q^{(G)}_{\infty}$ &
  $\|\text{Error}^{(\widetilde{G})}\|_h$ & $\|\text{Error}^{(\widetilde{G})}\|_{\infty}$
  & $q_h^{(\widetilde{G})}$ & $q^{(\widetilde{G})}_{\infty}$ 
  \\
\midrule
 16 & 2.49e-2 & 1.05e-1 &  -  &  -     & 2.44e-2 & 1.06e-1 &  -  &  -   \\       
 32 & 2.88e-3 & 1.49e-2 & 3.11 & 2.81  & 2.00e-3 & 1.56e-2 & 3.61 & 2.76  \\     
 64 & 4.03e-4 & 2.89e-3 & 2.84 & 2.37  & 2.31e-4 & 2.39e-3 & 3.11 & 2.70  \\     
 128 & 6.32e-5 & 6.89e-4 & 2.67 & 2.07 & 3.59e-5 & 4.16e-4 & 2.69 & 2.52  \\    
 256 & 1.06e-5 & 1.77e-4 & 2.58 & 1.96 & 6.24e-6 & 9.84e-5 & 2.53 & 2.08  \\    
\bottomrule
\end{tabularx}
\caption{Errors and convergence rates for two discretizations of the contravariant
        metric tensor. Results for the \revisedgreen{metric tensor
        discretization}
        defined in
        (\ref{stable_G}) are labeled $(G)$ and the \revisedgreen{modified discretization} defined in
        (\ref{G_mod}) are labeled $(\widetilde{G})$. Here, $q_h$ and
        $q_{\infty}$ denote the convergence rates in the $l_2$ and $l_{\infty}$
        norms, respectively.}
\label{t:test1_G}
\end{table}
Table \ref{t:test1_G} lists the errors and convergence rates for the energy positivity preserving
\revisedgreen{discretization} of $G$ in (\ref{stable_G}) and the modified $\widetilde{G}$ in (\ref{G_mod}). In most
cases, the error is smaller for the scheme using $\widetilde{G}$ compared to $G$. Note that
the convergence rates for the velocity components on the finer grids seem to drop to order $2.5$ in
$l_2$.

\subsection{Characteristic boundary error}
In this section, we continue analyzing the error from the experiment in Section~\ref{sec_exp1} with
the modified metric tensor \revisedgreen{discretization} (\ref{G_mod}). In particular, we observed that the convergence rate
dropped to order $2.5$ in the $l_2$ norm on the finer grids.  This convergence rate is consistent
with the truncation error analysis using the energy method, see \cite{wang2017}, and generalizes for
higher order operators as $p + 1/2$ where $p$ is the order of the truncation error of the SBP
operator near the boundary.  According to the general theory for SBP finite difference
approximations of first order hyperbolic problems~\cite{Gustafsson-75}, the convergence rate is in
many cases one order higher than the order of the truncation error at the
boundary, \revisedgreen{measured in the $l_2$ norm}. However, as is pointed out in \cite{Petersson2018}, the general theory assumes a
non-characteristic boundary. Since the acoustic wave equation contains zero-speed characteristic
variables, the general theory does not apply.

In order to examine the error in more detail, we symmetrize and diagonalize the
acoustic wave equation with respect to the direction normal to the boundary, see Appendix
\ref{a:symmetrization}. For the bottom boundary, we obtain the characteristic
variables
\begin{align}
\begin{aligned}
        w^+ &= 
        \frac{1}{\sqrt{2 J g_{11}}}
        \left(
        J \sqrt{g_{11}} p + \eta v^2
        \right)
        , \quad
        w^- = 
        \frac{1}{\sqrt{2 J g_{11}}}
        \left(
        -J \sqrt{g_{11}} p + \eta v^2
        \right)
        , 
        \\
        w^0 &= \frac{1}{\sqrt{2 J g_{11}}} \left(g_{11}v^1 + g_{12}v^2\right)
                  = \frac{v_1}{\sqrt{2 J g_{11}}}, 
        \label{characteristics}
\end{aligned}
\end{align}
where $\eta  = \sqrt{g_{11}g_{22} - g_{12}^2}$. These variables correspond to
the eigenvalues 
\[
\Lambda = \frac{\sqrt{g_{11}}}{\eta}\,
\begin{bmatrix}1 & & \\ & -1 & \\ & & 0 \end{bmatrix}.
\]
The characteristic variable $w^+$ corresponds to the positive eigenvalue and is "incoming", meaning
that it propagates into the domain, whereas $w^-$ corresponds to the negative eigenvalue and is
"outgoing", implying that it propagates out of the domain. However, the eigenvalue corresponding to
the variable $w^0$ is zero, making it boundary characteristic, meaning that it does not propagate in
the direction normal to the boundary. It is the presence of the boundary characteristic variable
that violates the assumptions of the general theory \cite{Gustafsson-75}.

To examine the error in the characteristic variables, we focus on a cross-section that starts in the
middle of the bottom boundary ($r^1 = 0.5$) and extends $10\%$ into the domain, i.e., $0 \leq r^2
\leq 0.1$. We use the SBP interpolation operators to compute the incoming and outgoing
characteristic variables at the cell-centers. Note that in
(\ref{characteristics}), only the boundary characteristic variable
$w^0$ depends on $v^1$. For simplicity, we therefore define
\begin{align}
        \|Err^{(c)}\| = \|\vb^1 - \tilde{\vb}^1\|, \
        \|Err^{(nc)}\| = \|\vb^2 - \tilde{\vb}^2\| + \|\pb - \tilde{\pb}\|
\end{align}
as the \revisedgreen{norm of the} error in the characteristic and non-characteristic boundary variables,
respectively. 

Figure \ref{fig:characteristic-error} shows the absolute value of the solution error at these grid
points for the two finest grids in Table~\ref{t:test1_G}. Note that the solution error at the first
four grid points is much larger than the error at the interior grid points. These four grid points
correspond to the region in which the SBP operators have a second order accurate truncation
error. Table \ref{t:char} lists the errors in the $l_2$ norm and $l_{\infty}$ norm measured at these
grid points for each grid. The convergence rates of the non-characteristic variables show close to
third order convergence in $l_2$, which is consistent with the standard theory
\cite{Gustafsson-75}. The convergence rate of the boundary characteristic variable is reduced
compared to the non-characteristic boundary variables. This drop in the convergence rate can also be
seen in Figure \ref{fig:characteristic-error}.

\begin{figure}[!htbp]
  \centering
  \subfloat[Characteristic boundary variable]
  {\includegraphics[scale=0.55]{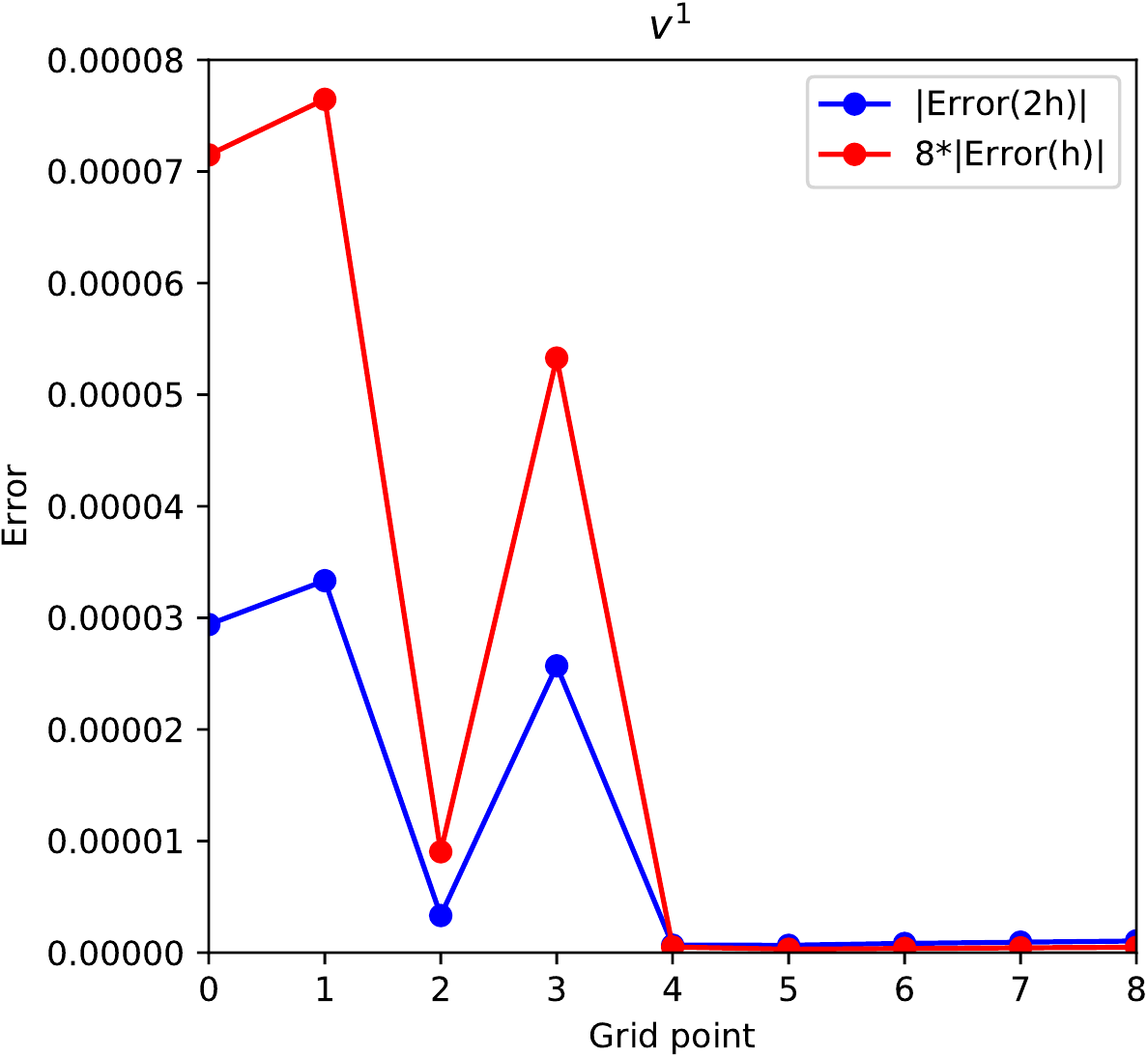}}
  \subfloat[Non-characteristic boundary variable]
  {\includegraphics[scale=0.55]{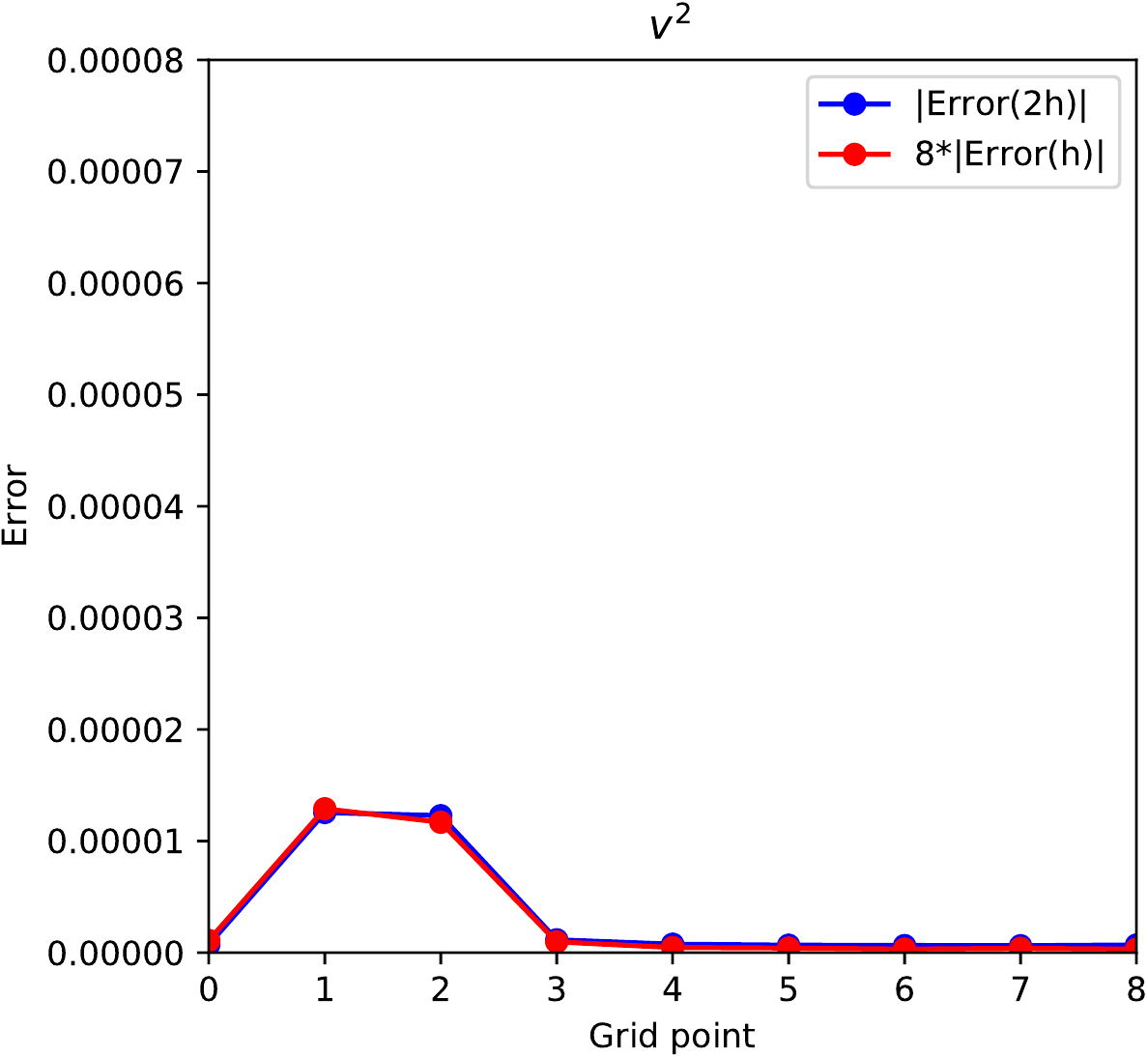}}
  \caption{The absolute error obtained in the characteristic variables for a few grid points normal to
  the bottom boundary on the two finest grids. Third order convergence is
  obtained where the error curves coincide.}
  \label{fig:characteristic-error}
\end{figure}  

\begin{table}[!htbp]
\centering
 \begin{tabularx}{\textwidth}{lllllllll}
  \toprule
  h & $\|\text{Err}^{(c)}\|_h$ & $\|\text{Err}^{(c)}\|_{\infty}$ & $q_h^{(c)}$ &
  $q_{\infty}^{(c)}$ &
  $\|\text{Err}^{(nc)}\|_h$ & $\|\text{Err}^{(nc)}\|_{\infty}$ & $q_h^{(nc)}$ &
  $q_{\infty}^{(nc)}$ \\
  \midrule
  6.25e-2 & 3.81e-3 & 9.82e-3 &  -  &  -  & 9.73e-3 & 2.43e-2 &  -  &  -   \\
  3.12e-2 & 6.12e-4 & 1.27e-3 & 2.64 & 2.95 & 1.10e-3 & 4.27e-3 & 3.14 & 2.51  \\
  1.56e-2 & 9.49e-5 & 3.95e-4 & 2.69 & 1.69 & 1.15e-4 & 7.00e-4 & 3.26 & 2.61  \\
  7.81e-3 & 1.29e-5 & 1.14e-4 & 2.87 & 1.79 & 1.41e-5 & 1.17e-4 & 3.02 & 2.58  \\
  3.91e-3 & 2.31e-6 & 3.05e-5 & 2.48 & 1.90 & 2.04e-6 & 2.04e-5 & 2.80 & 2.52\\ 
  \bottomrule
\end{tabularx}
\caption{Errors and convergence rates for boundary characteristic variables (c)
        and non-boundary characteristic variables (nc) are reported. Errors are
        measured for the grid points in
the SBP boundary-modified region of the bottom boundary.}
\label{t:char}
\end{table}

\subsection{Rotational invariance}
A desirable property of the numerical scheme is that its accuracy and stability should not depend on
the orientation of the coordinate system. Since we have chosen to decompose the velocity field with
respect to the covariant basis, the discretization is local with respect to the orientation of the
curved grid lines, and therefore rotationally invariant. If we instead the
decompose the velocity field with respect to the Cartesian basis, the
discretization does not preserve rotational invariance. In this section we compare the accuracy of these two
approaches.

When the velocity is decomposed with respect to the Cartesian basis, the metric tensor that
transforms from Cartesian components to contravariant components is defined by
\begin{align}
        \begin{bmatrix}
                v^1 \\
                v^2
        \end{bmatrix}
        =
        \begin{bmatrix}
                A_{11} & A_{12} \\
                A_{21} & A_{22}
        \end{bmatrix}^T
        \begin{bmatrix}
                v_x \\
                v_y
        \end{bmatrix}.
\end{align}
This transformation is explicitly given by
\begin{align}
        A = 
        \begin{bmatrix}
                A_{11} & A_{12} \\
                A_{21} & A_{22}
        \end{bmatrix}
        =
        \begin{bmatrix}
        {\eb}_x \cdot {\ab}^1 & {\eb}_x \cdot {\ab}^2 \\
        {\eb}_y \cdot {\ab}^1 & {\eb}_y \cdot {\ab}^2 \\
        \end{bmatrix}.
\end{align}
When a grid is rotated with respect to the Cartesian basis vectors such that the vector pairs $(\mathbf{\eb}_x,\ {\ab}^1)$ and
$(\mathbf{\eb}_y,\ \ab^2)$ become orthogonal, the diagonal entries in $A$ vanish and only
off-diagonal entries remain. These off-diagonal entries must be treated using interpolation. It is
therefore to be expected that the Cartesian formulation will suffer from accuracy degradation for
such grids.

Before proceeding to the numerical experiment, we give some more details.
When transformed to the Cartesian basis, the governing equations become 
\begin{align*}
        \frac{\partial p}{\partial t} + \frac{1}{J}
        \frac{\partial}{\partial r^1}     (JA_{11}v_x + JA_{21}v_y)
        +
        \frac{\partial}{\partial r^2}      (JA_{12}v_x + JA_{22}v_y)
        = 0, \\
        \frac{\partial v_x}{\partial t} 
        + 
        A_{11}\frac{\partial p}{\partial r^1} + A_{12}\frac{\partial p}{\partial r^2} =0, \\
        \frac{\partial v_y}{\partial t} 
        + 
        A_{21}\frac{\partial p}{\partial r^1} + A_{22}\frac{\partial p}{\partial
        r^2} =0. 
\end{align*}
The semi-discrete approximation (excluding SAT penalty terms) becomes
\begin{align}
\begin{aligned}
        \frac{d\pb}{dt} + \hat{J}^{-1}[\widehat{D}_1 \
        \widehat{D}_2]H^{-1}\tilde{A}^THJ
        \begin{bmatrix}
                \vb_x \\
                \vb_y
        \end{bmatrix} 
        &= 0,
        \\
        \frac{d}{dt}
        \begin{bmatrix}
                \vb_x \\
                \vb_y
        \end{bmatrix}
        + \tilde{A}
        \begin{bmatrix}
                D_1 \pb \\
                D_2 \pb
        \end{bmatrix}
        &=
        \begin{bmatrix}
                0 \\
                0
        \end{bmatrix}.
\end{aligned}
        \label{cartesian}
\end{align}
When discretizing the block matrix $\tilde{A}$ in (\ref{cartesian}), it is necessary to use interpolation operators for the
off-diagonal entries, 
\begin{align}
        \tilde{A} = 
        \begin{bmatrix}
                \tilde{A}_{11,1} & \tilde{A}_{12,1}P_{12} \\
                \tilde{A}_{21,2}P_{21} & \tilde{A}_{22,2}
        \end{bmatrix},
\end{align}
where the diagonal matrix $\tilde{A}_{ij,k}$ is obtained by discretizing entry $A_{ij}$
in $A$ on the edge-$k$ grid.
Here, the interpolation operators are $P_{12} = (P \otimes \hat{P})$ and $P_{21} =
(\hat{P}
\otimes P)$. 

When the velocity field is decomposed with respect to the Cartesian basis, the mixed terms that
arise in the acoustic wave equation depend on the angle between the contravariant basis and
Cartesian basis vectors. Since the mixed terms must be discretized by a combination of interpolation
and differentiation operators, the resulting operators become wide. These operators may be less
accurate than the staggered difference operators used for the diagonal terms.

To test how the accuracy depends on these angles, we use the manufactured solution (\ref{mms}) for
the following disc geometry with a cavity,
\begin{align}
\begin{aligned}
 x(r^1,r^2) &= \xi(r^1)\cos(\phi + r^2),
 \\
 y(r^1,r^2) &= \xi(r^1)\sin(\phi + r^2),
\end{aligned}
\nonumber
\end{align}
where the radial coordinate is stretched using linear interpolation from the
inner radius $R_0$ to the outer radius $R_1$ by the function
\[
        \xi(r^1) = (R_1 - R_0)r^1(ar^1 + (1-a)) + R_0.
\]
The stretching parameter $a$ is taken to be $a = 4\pi/n_2$ ($n_2$ is the number of cells in the
angular, $r^2$, direction). The inner and outer radii are set to $R_0 = 0.3$ and $R_1=1$. Periodic
boundary conditions are used in the angular direction and an offset $\phi = 0.2\pi$ is used to avoid
aligning the periodic boundary with any of the Cartesian axes.

On the coarsest grid we use $n_1 = 16$ grid cells in the radial direction and $n_2 = 48$ grid cells
in the angular direction. Finer grids are obtained by a factor of two grid refinement at each
level. The solution is advanced in time until $T = 0.5$ using the time step $\Delta t =
0.015625$. The results in Table \ref{t:disc} show that the scheme using the covariant representation
of the velocity field is always more accurate compared to the Cartesian basis decomposition.
\begin{table}[!htbp]
\centering
 \begin{tabularx}{\textwidth}{lllllllll}
  \toprule
  N & $\|\text{Err}^{(co)}\|_h$ & $\|\text{Err}^{(co)}\|_{\infty}$ &
  $q_2^{(co)}$ &
  $q_{\infty}^{(co)}$ &
  $\|\text{Err}^{(ca)}\|_h$ & $\|\text{Err}^{(ca)}\|_{\infty}$ & $q_h^{(ca)}$ &
  $q_{\infty}^{(ca)}$ \\
  \midrule
  16 & 1.69e-2 & 4.78e-2 &  -  &  -      & 6.65e-2 & 2.71e-1 &  -  &  -   \\    
  32 & 6.74e-4 & 2.43e-3 & 4.18 & 3.87   & 2.72e-3 & 1.28e-2 & 4.15 & 3.96  \\  
  64 & 4.68e-5 & 3.29e-4 & 3.67 & 2.75   & 2.04e-4 & 1.41e-3 & 3.56 & 3.04  \\  
  128 & 5.59e-6 & 6.67e-5 & 2.99 & 2.25  & 2.68e-5 & 3.20e-4 & 2.86 & 2.09  \\ 
  256 & 8.62e-7 & 1.55e-5 & 2.67 & 2.08  & 4.41e-6 & 7.87e-5 & 2.57 & 2.00 \\
  \bottomrule
\end{tabularx}
\caption{Errors and convergence rates obtained for the disc geometry with a
        cavity.  Results for the covariant (co) and Cartesian (ca) formulations
        are reported separately.}
\label{t:disc}
\end{table}

\subsection{Point source on the boundary}
In computational acoustics, point sources are frequently used to model explosions.  Spurious
oscillation can arise when source terms are introduced in schemes discretized by central finite
differences.  The triggering of spurious oscillations can be avoided by using Cartesian staggered
grids, or by smoothing out the source discretization by imposing smoothness conditions, see
\cite{Petersson-etal-source-16}.  However, there is no guarantee that spurious oscillations are
avoided in our curvilinear staggered scheme.

The point source we consider is placed on the boundary $r^2=1$ by enforcing the pressure boundary condition
\begin{align}
        p(r^1,1,t) = \frac{1}{\sqrt{(\partial x/\partial r^1)^2 + (\partial
        y/\partial r^1)^2}}\delta(r^1 - r_*)s(t).
\label{delta_fcn}
\end{align}
In the above, $\delta(r^1-r_*)$ is the Dirac distribution, centered at the source location $r_*$,
and $s(t)$ is a given source time function. 
For a detailed discussion on how to discretize the
Dirac distribution for hyperbolic problems, see \cite{Petersson-etal-source-16}. To place the source
on the boundary, we treat (\ref{delta_fcn}) as a boundary condition and impose it using a SAT term,
see \cite{o2017energy} for details.

In this test, only the top boundary is curved, and its shape is given by the
Gaussian hill,
\[
        x(r^1,1) = 10r^1, \quad
        y(r^1,1) = 5 + e^{-\frac{(r^1-0.5)^2}{\sigma^2}}, 
\]
where \revisedgreen{$\sigma = 0.105$}. A 1-D coordinate stretching is defined to generate the mapping. The source is
placed on the Gaussian hill at the location $\rb_s= (0.45, 1)$ in parameter space, or $\xb_s = (4.5,
5.79711)$ in physical space, see Figure \ref{fig:point-source-simulation}. The source time function
is a Ricker wavelet, i.e.,
\[
 s(t) = (1 - 2\pi^2(t - t_0)^2)e^{-\pi^2(t - t_0)^2},
\]
where the time delay is $t_0 = 1.7$, to prevent an abrupt startup. We define the minimum wavelength
that must be resolved as $\lambda_{min} = c_{min}/f_{max}$, where $c_{min} = 1$ is the minimum wave
speed (here set to \revisedred{1} due to solving the acoustic wave equation in dimensionless form) and
$f_{max} = 2.5$. This estimate corresponds to the frequency where the Ricker wavelet reaches $5\%$
of the peak amplitude in the frequency domain.

Since we do not have an analytic solution for
this test, we compare the numerical solutions against a reference solution that has
been computed on a very fine mesh. \revisedred{When estimating convergence
rates, we take into account that the reference solution itself contains an
error.} The resolution
of the reference solution is \revisedred{$4096 \times 2048$} grid points, which corresponds
to \revisedred{$N_{\lambda} = \lambda_{min}/h \approx 160$} grid points per shortest wavelength. The error is defined
as the difference between
the reference solution and a numerical solution measured in time at the interior
receiver
location $\rb_r = (0.5, 0.5)$, corresponding to $\xb_r = (5, 3)$.
Since this point does not align with a
grid point for each field, we use cubic interpolation of the neighboring values
to estimate the solution.

We
advance the solution in time using the time step $\Delta t = 0.03125$ (for the coarsest
grid) until the final time $T = 7.8125$. Figure
\ref{fig:point-source-simulation} shows that there are no visible artifacts from
placing a point source on the boundary \revisedred{when discretizing using the
modified metric tensor (\ref{G_mod}). On the other hand, the provably stable
metric tensor discretization (\ref{stable_G}) causes spurious oscillations.
} For the same simulation, we also show the pressure as function of time
at the receiver location and the estimated error (Figure
\ref{fig:time-history}). The initial oscillations correspond to the direct
arrival of the pressure wave from the source; oscillations at later times are due to
reflections against the non-planar boundary.

\begin{figure}[!htbp]
  \centering
  \subfloat[$t=2.25$ ]
  {\includegraphics[scale=0.35]{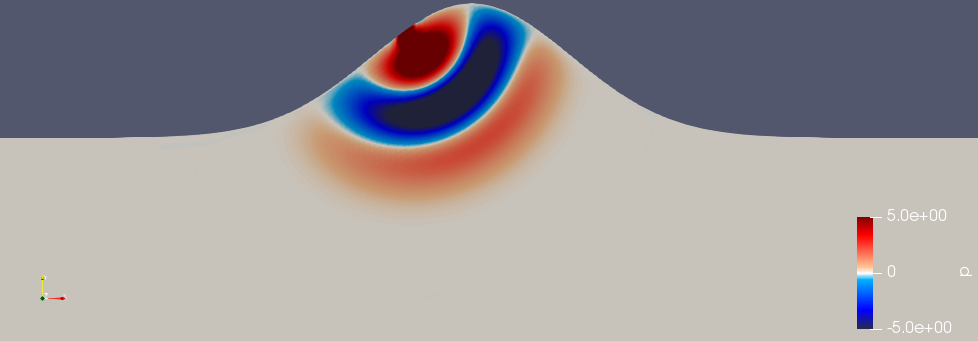}}

  \subfloat[$t=3.375$]
  {\includegraphics[scale=0.35]{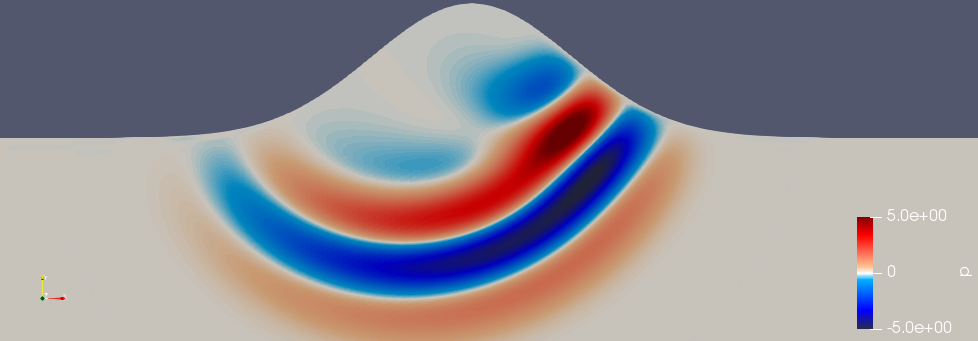}}
  \caption{Point source placed on top of a curved Gaussian
          hill at $\xb_s = (4.5, 5.79711)$. The pressure wave field is shown at two different instances
          in time.}
  \label{fig:point-source-simulation}
\end{figure}  

\begin{figure}[!htbp]
        \centering
        \subfloat[]
        {\includegraphics[scale=0.55]{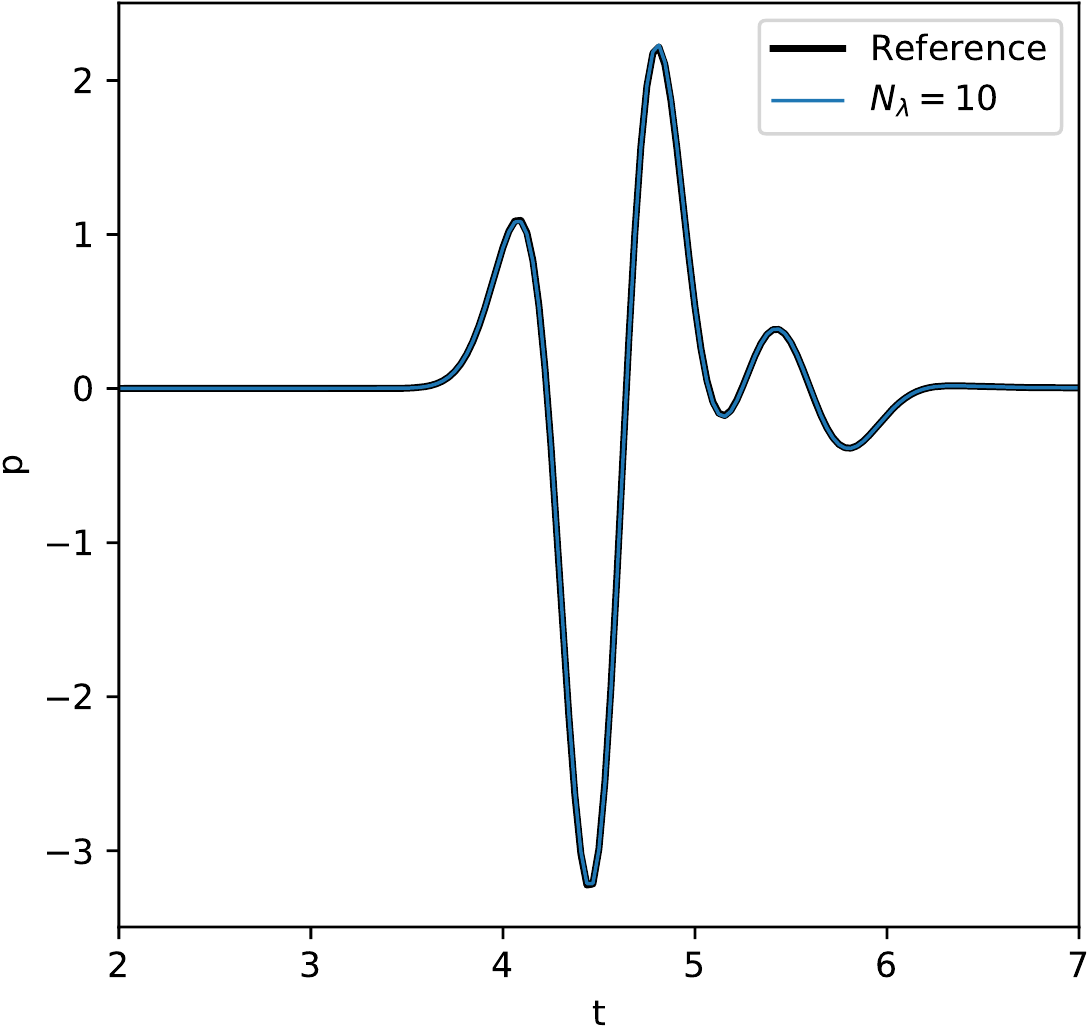}}
        \subfloat[]
        {\includegraphics[scale=0.55]{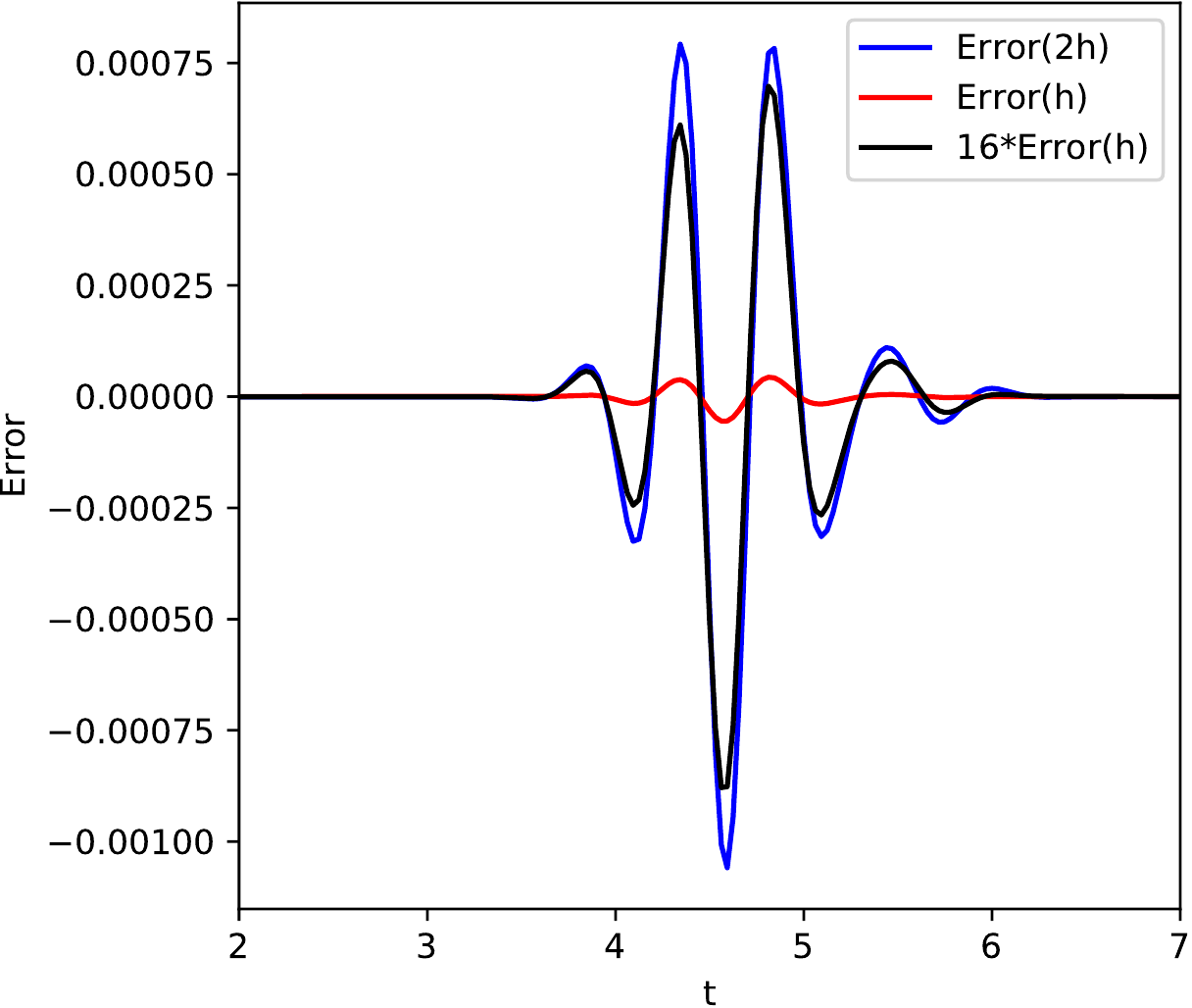}}
        \caption{Point source placed on top of a curved Gaussian boundary 
        at $\xb_s = (4.5, 5.79711)$. 
        (a) Pressure as a function of time at $\xb_r = (5,3)$ for
        a numerical solution \revisedgreen{with} $10$ grid points per minimum
        wavelength compared to a numerical solution 
        \revisedred{with $160$} grid points per
minimum wavelength. (b) Error in the pressure on the second two finest grids using the
numerical solution on the finest grid as reference solution. The black line
        shows \revisedred{16} times the error shown by the red line.
This scaling corresponds to fourth order convergence.
}
        \label{fig:time-history}
\end{figure}

\begin{table}
\revisedred{
\begin{tabularx}{\textwidth}{XXXXXXX}
\toprule
$N_{\lambda}$ & $\|\Delta p\|_{\rb}$ & $q^{(p)}_h$ &
$\|\Delta v^1\|_{\rb}$ & $q^{(v^1)}_h$ &
$\|\Delta v^2\|_{\rb}$ & $q^{(v^2)}_h$ \\
\midrule
5 & 1.20e-01 & 2.75 & 9.93e-02 & 3.14 & 1.20e-01 & 2.77 \\ 
10 & 6.32e-03 & 4.25 & 4.77e-03 & 4.38 & 6.40e-03 & 4.22 \\ 
20 & 3.16e-04 & 4.32 & 2.52e-04 & 4.24 & 3.31e-04 & 4.28 \\ 
40 & 1.64e-05 & 4.26 & 1.45e-05 & 4.12 & 1.77e-05 & 4.22 \\ 
\bottomrule
\end{tabularx}
}
\caption{Relative errors and convergence rates obtained for a point source placed on top of a curved boundary.
        Errors are computed using a numerical reference solution
        \revisedred{with $N_{\lambda}
        = 160$} grid points per minimum wavelength.}
\label{tbl:source}
\end{table}

To investigate the error in more detail, let $\|\Delta p\|_{\rb}$ denote the
error in pressure measured in the $l_2$ norm over time at the fixed point $\rb$,
and let the two velocity components be defined in a similar manner. Table
\ref{tbl:source} lists the relative error and convergence rates for both
pressure and velocity as well as an estimate of the number of grid points
per shortest wavelength to obtain the desired error. We note that fourth order
convergence rates are obtained for all components. \revisedgreen{In this case,} we see no
indication that the error due to the characteristic boundary variable has
polluted the solution at this interior location.

\subsection{Computational Efficiency}
After having established that the staggered scheme with the
modified metric tensor is capable of handling point sources on the boundary, we
now investigate its computational performance and compare it to a collocated
finite difference method. This test case is quite similar to the previous test,
and can be reproduced by using the code found in
the repository github.com/ooreilly/sbp/.

In the collocated scheme the governing equations are discretized using the classical fourth order
SBP operators derived by B. Strand in \cite{Strand-94}. One can convert the staggered scheme
into a collocated scheme by replacing the interpolation operators with identity
matrices and the derivatives with Strand's SBP operators.

The primary difference in terms of computational work between the collocated and
staggered scheme stems from the use of interpolation. In the interior of the staggered scheme, the use of
interpolation operators in the velocity equations introduces two stencils with
$4 \times 7$ points, whereas the remaining terms involve only four stencils
 with $4$ points each. In contrast, the entire collocated scheme consists of only
six stencils with 4 points each. Despite these large differences in computational work, the 
staggered scheme can outperform the collocated scheme when running on modern hardware. 

The performance of each scheme can be highly dependent on the implementation.
For prototyping purposes, it can be quite convenient to implement the schemes as
sparse matrix and dense vector products. In this way, the implementation 
looks similar to the discretizations presented in Section
\ref{s:operators}. Unfortunately,
this approach results in excessive computational overhead due explicitly storing
the large sparse matrix that holds the spatial discretization. Instead,
significant performance improvements can be achieved if the implementation is
matrix-free, or stencil-based. 

We measure the performance of the collocated and staggered scheme implemented in
a stencil-based manner. Both these compute kernels are executed on a Nvidia
Geforce RTX 2080 Ti card in single precision. In our timings, we only
investigate the performance of the interior computation (in most practical
cases, the boundary only consitutes a small fraction of the total cost). No
particular effort has been invested into optimizing the kernels performance.
Both kernels are memory bound, achieving $75 \%$ of peak memory bandwidth
(staggered) and $85 \%$ of peak memory bandwidth (collocated). Compute
utilization of the GPU streaming multiprocessors (SM)s reaches $45 \%$
(staggered) and $25 \%$ (collocated). This performance analysis was conducted
using a grid size of 1024 x 512 grid points.

\begin{figure}
        \centering
        \includegraphics[width=0.7\textwidth]{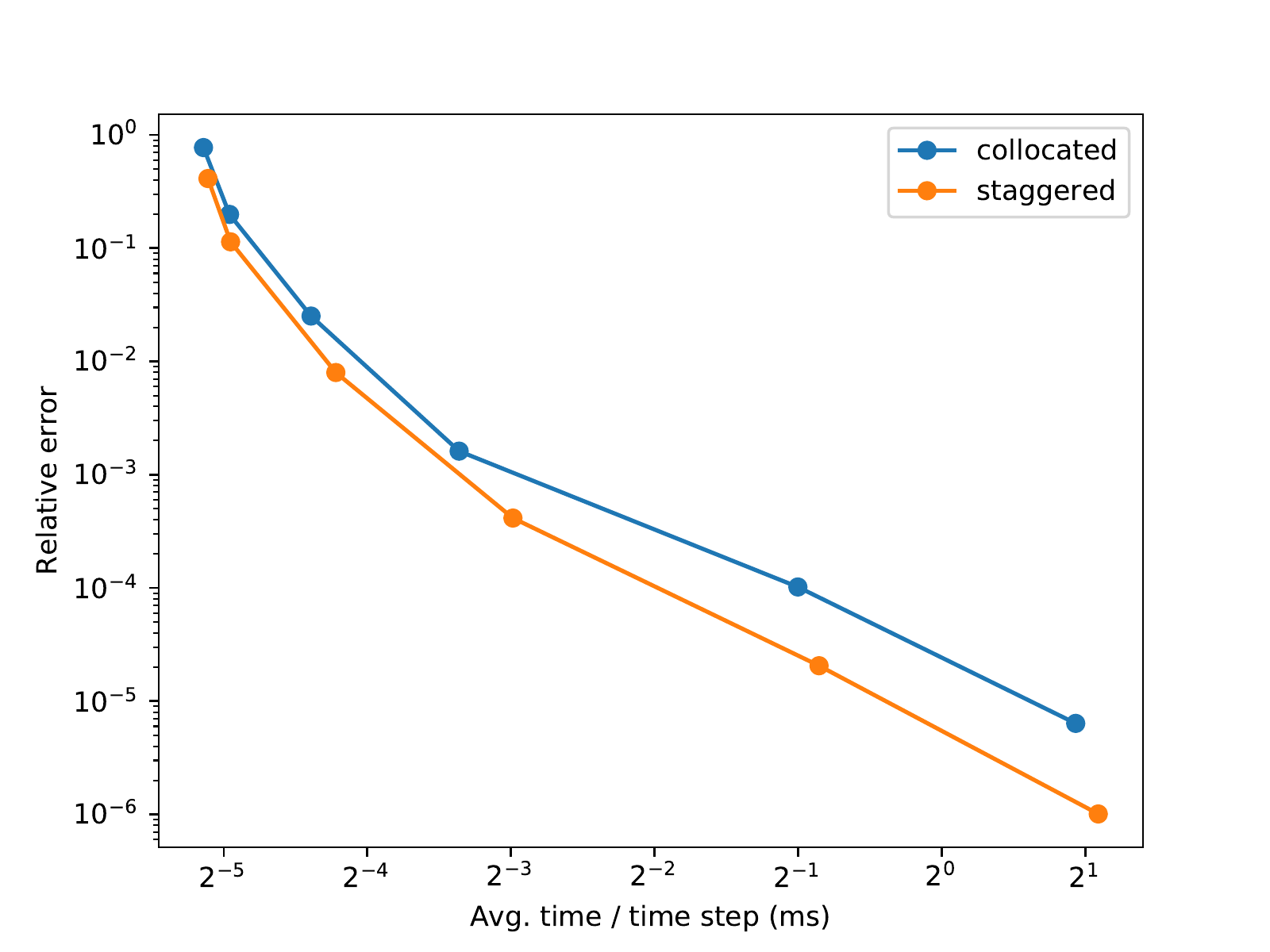}
        \caption{
                Computational efficiency of the SBP staggered scheme (this
        work) and collocated SBP scheme. The test uses a curved Gaussian
        boundary with a point source placed in the interior of the domain, at
        $\xb_s = (5, 3)$. The error is measured as a function of time at the
        receiver location $\xb_r = (4, 2.16149)$. }    
        \label{fig:efficiency}
\end{figure}

\revisedblue{
The solution is advanced in time using the same settings as before. Before the
final time is reached, acoustic waves have reflected against the exterior
boundaries and these reflections have reached the receiver location.  As can be
seen in Figure \ref{fig:efficiency}, the staggered scheme outperforms the
collocated scheme on all grids tested, and its performance improves as the error
levels decrease. On the finest grids, the difference in relative error approaches an
order of magnitude.
}

\section{Conclusions}\label{s:conclusions} 
We have shown how to construct a provably stable high order accurate finite difference approximation
of the acoustic wave equation in covariant form on general curvilinear staggered grids. To obtain an
energy conserving scheme, the discretization uses SBP interpolation and differentiation operators and
imposes the boundary conditions weakly. After transforming the acoustic wave equation to curvilinear
coordinates, the covariant metric tensor emerges in the kinetic energy. To make the scheme stable,
the discretized kinetic energy must form a discrete norm of the dependent variables. The weights in
the norm are represented by a symmetric matrix that depends on the discretization of the metric
tensor. Stability follows if this matrix is positive definite. 
We offer two alternative ways of making \revisedgreen{this matrix} positive definite. The first option uses
interpolation of both the diagonal and off-diagonal terms in the metric tensor. While this approach
leads to a stable scheme for all non singular meshes, it results in a slight degradation in accuracy
and is more computationally expensive compared to our second alternative. In that approach, a
modified discretization of the metric tensor is derived where only the off-diagonal elements are
interpolated. Here, positive definiteness of the norm matrix depends on the mesh quality. Estimates
are theoretically derived to efficiently determine if the discretization is stable for a given
mesh. Numerical experiments \revisedgreen{illustrate} the benefits of using interpolation operators with norms
close to \revisedred{one} and demonstrate the accuracy and stability of the proposed methodology.

While this work has focused on the acoustic wave equation in covariant form, future extensions to
Maxwell's equations or the elastic wave equation should also be possible.
\revisedblue{However, these extensions might be non-trivial and warrant further
investigations. For Maxwell's equations one must contend with the discretization of
the curl operator in a stable and consistent manner.} Some additional
difficulties with solving the elastic wave equation in covariant form are the emergence of
Christoffel symbols, and transformation of the stress tensor from covariant to
contravariant form.
 
\section*{Acknowledgment} 
We would like to thank Dr.~Martin Almquist for  many fruitful discussions during the initial
developments of this work, as well as his insightful comments and feedback on an early version
of this manuscript. 

This research was supported by the Southern California Earthquake Center (Contribution No. 9223),
with contributions from NSF Cooperative Agreements EAR-1600087, USGS Cooperative Agreement
G17AC00047, NSF-ACI Award 1450451, and Pacific Gas and Electric.

This work performed under the auspices of the U.S. Department of Energy by Lawrence Livermore
National Laboratory under Contract DE-AC52-07NA27344. This is contribution LLNL-JRNL-780017.

\begin{appendix}
\section{Weakly imposed pressure boundary condition}\label{a:weak}
Here, we demonstrate how to weakly impose the homogeneous pressure boundary condition method in the
covariant formulation of our scheme. To simplify the analysis, we will only consider the left
boundary $r^1 = 0$ (the other boundaries are treated in a analogous manner).

When the grid is not periodic, the SBP property (\ref{SBP_property_2d}) applied to the right-hand
side of the energy rate (\ref{energy_rate}) leads to
\begin{align}
  \frac{d\mathcal{E}}{dt} = 
  (\vb^1_L)^T \hat{J}_L\hat{M}_2 \pb_L, 
  \nonumber
\end{align}
where $\hat{J}_L$ contains the values of the Jacobian along the
edge-1 grid on left
boundary, and the contributions from the other boundaries have been
neglected. Unless a boundary condition is specified, the energy
rate is indefinite. For simplicity, we impose the homogeneous boundary condition
$p=0$ using the Simultaneous Approximation Term (SAT) penalty method.

Consider the velocity equations (\ref{eq:velocity_discrete}), augmented with a SAT penalty term on
the right-hand side,
\begin{align}
        \frac{d}{dt}
        \begin{bmatrix} 
                \vb^1 \\
                \vb^2
        \end{bmatrix}
+ G \begin{bmatrix}
                D_1 \pb \\
                D_2 \pb
            \end{bmatrix}
            = G \begin{bmatrix}
                    \Sb \\
                    0
                \end{bmatrix}.
        \nonumber
\end{align}
The penalty vector $\Sb$ restricts the pressure field to the left boundary, and can
be written as 
\begin{align}
        \Sb = -H_1^{-1}J_1^{-1}L\hat{M}(J_1)_L(\pb_L  -\fb_L(t)), 
        \label{penalty}
\end{align}
where the matrix $L$ is defined such that $L^T\vb^1 = \vb^1_L$, and
$\fb_L(t)$ is the boundary data. Since we focus on homogeneous boundary
conditions, we set $\fb_L = 0$.
When computing the energy rate including the penalty vector $\Sb$ in
(\ref{penalty}), we find
\begin{align}
        \frac{d\mathcal{E}}{dt}
        &= 
        (\vb^1_L)^T \hat{J}_L\hat{M} \pb_L 
        -
        \begin{bmatrix} 
                \vb^1 \\
                \vb^2
        \end{bmatrix}^T
        HJ
        \begin{bmatrix}
                    H_1^{-1}J_1^{-1}L\hat{M}(J_1)_L\pb_L \\
                    0
        \end{bmatrix}
        \nonumber
        \\
        &= 
        (\vb^1_L)^T \hat{J}_L\hat{M} \pb_L 
        -
        (\vb^1_L)^T \hat{J}_L\hat{M} \pb_L 
        = 0.
        \nonumber
\end{align}
By repeating the same procedure for the remaining
boundaries, the energy of the semi-discrete approximation is conserved.

\section{Symmetrization and diagonalization of the governing equations}\label{a:symmetrization}
We show that the 2D acoustic wave equation in covariant form (\ref{eq:pressure})-(\ref{eq:velocity})
 can be written as a symmetric hyperbolic
system. Recall the governing equations,
\begin{align}
  J\frac{\p p}{\p t} + \frac{\p}{\p r^1}\left(Jv^1 \right) + \frac{\p}{\p
    r^2 }\left(Jv^2\right) &= 0,\label{eq:pressure2} \\
  \frac{\p }{\p t}
        \begin{bmatrix}
        v^1 \\
        v^2
        \end{bmatrix}
        +
        \begin{bmatrix}
                g^{11} & g^{12} \\
                g^{12} & g^{22}
        \end{bmatrix}
        \begin{bmatrix}
        \frac{\p p}{\p r^1} \\
        \frac{\p p}{\p r^2}
        \end{bmatrix}
                 = \begin{bmatrix}
                         0 \\
                         0
                        \end{bmatrix}. 
                        \label{eq:velocity2}
\end{align}
Multiplying \eqref{eq:velocity2} by the covariant metric tensor gives
the system
\begin{align}
  J\frac{\p p}{\p t} + \frac{\p}{\p r^1}\left(Jv^1 \right) + \frac{\p}{\p
    r^2 }\left(Jv^2\right) &= 0,
        \nonumber
         \\
   J^{-1}
  \begin{bmatrix}
                g_{11} & g_{12} \\
                g_{12} & g_{22}
  \end{bmatrix} 
  \frac{\p }{\p t}
  \begin{bmatrix}
    Jv^1 \\
    Jv^2
  \end{bmatrix}
  +
  \begin{bmatrix}
    \frac{\p p}{\p r^1} \\
    \frac{\p p}{\p r^2}
  \end{bmatrix}
  = \begin{bmatrix}
    0 \\
    0
  \end{bmatrix}. 
        \nonumber
\end{align}
This system can be written in matrix form as
\begin{equation}
W \frac{\p \qb}{\p t} + A \frac{\p \qb}{\p r^1} + B \frac{\p \qb}{\p r^2} = 0,\quad \qb
= \begin{bmatrix} p\\ J v^1 \\ Jv^2 \end{bmatrix},
        \label{symmetrized_system}
\end{equation}
where
\begin{equation}
  \quad W =
  \begin{bmatrix}
    J & 0 & 0 \\
          0 & J^{-1}g_{11} & J^{-1}g_{12} \\
          0 & J^{-1}g_{12} & J^{-1}g_{22} 
  \end{bmatrix}
  \quad A=
  \begin{bmatrix}
    0 & 1 & 0 \\
    1 & 0 & 0 \\
    0 & 0 & 0 
  \end{bmatrix}
  \quad B=  \begin{bmatrix}
    0 & 0 & 1 \\
    0 & 0 & 0 \\
    1 & 0 & 0 
  \end{bmatrix}.
        \nonumber
\end{equation}
The matrices $A$ and $B$ are symmetric and $W=W^T$ is positive definite. 
Therefore, (\ref{symmetrized_system}) is a symmetric hyperbolic system. The energy in the system is
given by
\begin{equation}
e(t) = \frac{1}{2} \int_\Omega \qb^T W \qb \, dr^1 dr^2, 
\end{equation} and is the same as the energy (\ref{continuous_energy}). 

Next, we compute the characteristic variables for (\ref{symmetrized_system}) with respect to the
direction normal to the boundary.  First, we apply a transform to remove $W$ in front of the time
derivative term in (\ref{symmetrized_system}). Note that since $W$ is positive definite, it can be
factorized into $W = FF^T$ using the Cholesky factorization. The matrix $F$ is the unique lower
triangular matrix given by
\begin{align}
        F = \frac{1}{\sqrt{Jg_{11}}}
        \begin{bmatrix}
                J\sqrt{g_{11}} & 0 & 0 \\
                0 & g_{11} & 0 \\
                0 & g_{12} & \eta
        \end{bmatrix}, \quad
        F^{-1} = \frac{1}{\eta \sqrt{Jg_{11}}}
        \begin{bmatrix}
                \eta \sqrt{g_{11}} & 0 & 0 \\
                0 & \eta J & 0 \\
                0 & -J g_{12} & J g_{11}
        \end{bmatrix},
\end{align}
where $\eta = \sqrt{g_{11}g_{22} - g_{12}^2} > 0$ ($\eta$ is positive since the metric tensor
$g_{ij}$ is positive definite). By inserting $W = FF^T$, $FF^{-1} = I$ and $F^{-T}F^T = I$, equation
(\ref{symmetrized_system}) is transformed to
\begin{align}
        F F^T \frac{\p \qb}{\p t} + A \frac{\p  F^{-T}F^T  \qb}{\p r^1} + B \frac{\p F^{-T}F^T \qb}{\p r^2} = 0.
        \label{symmetrized_system2}
\end{align}
We introduce the definitions,
\begin{align}
        \qb' = F^T \qb, \ A' = F^{-1} A F^{-T}, B' = F^{-1} B F^{-T}, 
        \nonumber
\end{align}
where
\begin{align}
        \qb' =
        \frac{1}{\sqrt{Jg_{11}}}  
        \begin{bmatrix}
                pJ\sqrt{g_{11}} \\
                g_{11}v^1 + g_{12}v^2 \\
                \eta v^2
        \end{bmatrix}, 
        \
        A' =
        \frac{1}{\sqrt{g_{11}}}
        \begin{bmatrix}
                0 & 1  & - g_{12}/ \eta \\
                1 & 0 & 0 \\
                -g_{12}/ \eta & 0 & 0
        \end{bmatrix}, 
        \
        B' = 
        \frac{\sqrt{g_{11}}}{\eta}
        \begin{bmatrix}
                0 & 0 & 1 \\
                0 & 0 & 0 \\
                1 & 0 & 0
        \end{bmatrix}.
        \nonumber
\end{align}
Then (\ref{symmetrized_system2}) becomes,
\begin{align}
        F \frac{\p \qb ' }{\p t} +  FA'F^T \frac{\p F^{-T}\qb}{\p r^1} + FB'F^T \frac{\p
        F^{-T}\qb '}{\p r^2} = 0.
        \nonumber
\end{align}
By applying the product rule, we find
\begin{align}
        \frac{\p \qb ' }{\p t} +  A' \frac{\p \qb'}{\p r^1} + B'\frac{\p \qb
        '}{\p r^2} + \left(A'F^T\frac{\p F^{-T}}{\p r^1} + B'F^T\frac{\p F^{-T}}{\p
        r^2}\right)\qb' = 0.
        \label{symmetrized_system3}
\end{align}
To investigate boundary errors near the bottom boundary $r^2 = 0$, where the
normal is $(n_1, n_2) = (0, -1)$, we consider $A'n_1 + B'n_2' = -B'$. Since $B'$
is symmetric, it has the eigendecomposition $B' = X_B\Lambda_BX_B^T$, where $X_B$ is an orthonormal eigenvector basis, and $\Lambda_B$ is a diagonal matrix. This eigendecomposition is given by 
\begin{align}
        X_B^T = 
        \frac{1}{\sqrt{2}}
        \begin{bmatrix}
                1 & 0 & 1 \\
                -1 & 0 & 1 \\
                0 & \sqrt{2} & 0
        \end{bmatrix}, \
        \Lambda_B = 
        \frac{\sqrt{g_{11}}}{\eta}
        \begin{bmatrix}
                1 & 0 & 0 \\
                0 & -1 & 0 \\
                0 & 0 & 0 
        \end{bmatrix}.
        \nonumber
\end{align}
Finally, in the direction $(n_1, n_2) = (0, -1)$, the characteristic properties
of (\ref{symmetrized_system3}) can be studied by the model problem,
\begin{align}
        \frac{\partial \qb'}{\partial t} +B'\frac{\partial \qb'}{\partial r^2} =
        0.
        \nonumber
\end{align}
Because $X_B$ is constant, we can change variables to $\wb = X_B^T\qb'$, 
which leads to the diagonal system
\begin{align}
        \frac{\partial \wb}{\partial t} + \Lambda_B\frac{\partial \wb}{\partial
        r^2} = 0.
        \nonumber
\end{align}
The characteristic variables are
\begin{align}
        \wb = 
        \begin{bmatrix}
                w^+ \\
                w^- \\
                w^0
        \end{bmatrix}
        = X_B^T\qb' = 
        \frac{1}{\sqrt{2 J g_{11}}}
        \begin{bmatrix}
                J \sqrt{g_{11}}p + \eta v^2 \\
                -J \sqrt{g_{11}}p + \eta v^2 \\
                \sqrt{2}(g_{11}v^1 + g_{12}v^2)
        \end{bmatrix}.
        \nonumber
\end{align}

\section{Solving the generalized eigenvalue problems for $B(\alpha,\beta)$}\label{app_eigenvals}
The conditions for the matrix $HJ\widetilde{G}$ to be positive definite relies on the solution of
two generalized eigenvalue problems \eqref{eq_geneig1} and \eqref{eq_geneig2}, see
Lemma~\ref{l:bound}. In this section we use Kronecker product identities to show how the eigenvalues
of the generalized eigenvalue problems can be calculated by instead solving a number of decoupled
1-D eigenvalue problems.

Due to the ordering of the 2-D dependent variables and SBP operators, we can directly use
Kronecker product identities to analyze the generalized eigenvalue problem \eqref{eq_geneig2}. The
analysis of \eqref{eq_geneig1} follows in a corresponding way after the dependent
variables and SBP operators have been permuted.

For all matrices $A$, $B$, $C$ and $D$ where the sizes are such that the products $AC$ and $BD$ are
well-defined, it is well-known that $(A\otimes B)(C \otimes D) = (AC) \otimes (BD)$. Furthermore,
for all matrices $A$ and $B$, we have $(A\otimes B)^T = A^T \otimes B^T$.

We proceed by analyzing the generalized eigenvalue problem \eqref{eq_geneig2}. To simplify the
notation in this section we drop the superscripts on the matrices ${\mathbb X}$ and ${\mathbb Y}$.
Because the matrix ${\mathbb X}$ is positive definite, the generalized eigenvalue problem
\[
 {\mathbb Y}\boldsymbol{y} = \lambda {\mathbb X}\boldsymbol{y},
\]
has the same eigenvalues as
\begin{equation}\label{eq_gen-eig-2}
{\mathbb X}^{-1/2}{\mathbb Y}{\mathbb X}^{-1/2}\widetilde{\boldsymbol{y}} = \lambda \widetilde{\boldsymbol{y}},
\end{equation}
and the eigenvectors are related by $\widetilde{\boldsymbol{y}} = {\mathbb X}^{1/2}\boldsymbol{y}$.

From the definitions of the 2-D SBP operators in Section~\ref{sec_2dsbp},
\[
{\mathbb Y}:= P_{c2}^T \hat{H} \hat{J} \hat{g}^{22} P_{c2} = (\hat{I}\otimes \hat{P}^T) (\hat{M}\otimes \hat{M})
\hat{J}\hat{g}^{22} (\hat{I}\otimes \hat{P})
= (\hat{M}\otimes \hat{P}^T \hat{M})  \hat{J}\hat{g}^{22} (\hat{I}\otimes \hat{P}).
\]
Denote the components of the norm weight matrix by $\hat{M} = \mbox{diag}(\hat{m}_0, \hat{m}_1,
\ldots, \hat{m}_{N+1})$. Furthermore, let the diagonal matrices $J_2 g_2^{22}$ and $\hat{J}_2
\hat{g}_2^{22}$ have the blocked matrix representations,
\[
J_2 g_2^{22} = \begin{bmatrix}
  K^{(2)}_0 & & & \\
  & K^{(2)}_1 & & \\
  & & \ddots & \\
  & & & K^{(2)}_{N+1}
\end{bmatrix},\quad
\hat{J}_2 \hat{g}_2^{22} = \begin{bmatrix}
  \hat{K}^{(2)}_0 & & & \\
  & \hat{K}^{(2)}_1 & & \\
  & & \ddots & \\
  & & & \hat{K}^{(2)}_{N+1}
\end{bmatrix}.
\]
In this case, $K^{(2)}_k$ is a diagonal $(N+1)\times(N+1)$ matrix and $\hat{K}^{(2)}_k$ is a
diagonal $(N+2)\times(N+2)$ matrix. They hold the values of the corresponding metric coefficients
along a nodal or a cell-centered grid line in the $r^2$-direction, with $r^1=\hat{x}_k$
(cf.~\eqref{eq_1d-grids} for a definition of $\hat{x}_k$).

From the definition of a Kronecker product,
\[
(\hat{M}\otimes \hat{P}^T \hat{M})  \hat{J}\hat{g}^{22} =
\begin{bmatrix}
  \hat{m}_0 \hat{P}^T \hat{M} \hat{K}^{(2)}_0 & & & \\
  & \hat{m}_1 \hat{P}^T \hat{M} \hat{K}^{(2)}_1 & & \\
  & & \ddots & \\
  & & &\hat{m}_{N+1} \hat{P}^T \hat{M}  \hat{K}^{(2)}_{N+1}
\end{bmatrix}.
\]
Therefore,
\[
  {\mathbb Y} =
  \begin{bmatrix}
    \hat{m}_0 \hat{P}^T \hat{M} \hat{K}^{(2)}_0\hat{P} & & & \\
    & \hat{m}_1 \hat{P}^T \hat{M} \hat{K}^{(2)}_1\hat{P} & & \\
    & & \ddots & \\
    & & &\hat{m}_{N+1} \hat{P}^T \hat{M}  \hat{K}^{(2)}_{N+1}\hat{P}
  \end{bmatrix}.
\]
In a similar way,
\[
  {\mathbb X} := H_2 J_2 g_2^{22} = (\hat{M}\otimes M)J_2 g_2^{22},
\]
and the above blocked representation of $J_2 g_2^{22}$ gives
\[
  {\mathbb X} =
  \begin{bmatrix}
    \hat{m}_0 M K^{(2)}_0 & & & \\
    & \hat{m}_1 M K^{(2)}_1 & & \\
    & & \ddots & \\
    & & & \hat{m}_{N+1} M K^{(2)}_{N+1} 
  \end{bmatrix}.
\]
Note that $\mathbb X$ is a positive definite diagonal matrix.

Because of the diagonal blocked structure of the matrices ${\mathbb X}$ and ${\mathbb Y}$, the
eigenvalue problem \eqref{eq_gen-eig-2} reduces to $N+2$ decoupled eigenvalue problems
\begin{equation}\label{eq_1d-eigen}
  {\mathbb B}_k\widetilde{\widetilde{\boldsymbol{y}}} = \lambda
  \widetilde{\widetilde{\boldsymbol{y}}},\quad k=0,1,\ldots, N+1,
\end{equation}
where,
\[
  {\mathbb B}_k = M^{-1/2} (K^{(2)}_k)^{-1/2} \hat{P}^T \hat{M} \hat{K}^{(2)}_k \hat{P} (K^{(2)}_k)^{-1/2} M^{-1/2}.
\]
Note that $(K^{(2)}_k)^{-1/2}$ and $M^{-1/2}$ commute because they are both diagonal. To summarize,
we have shown that the eigenvalues of the 2-D generalized eigenvalue problem \eqref{eq_geneig2}
can be calculated by solving $N+2$ decoupled 1-D eigenvalue problems \eqref{eq_1d-eigen}.

In a similar fashion, calculating the eigenvalues of the generalized eigenvalue problem \eqref{eq_geneig2}
reduces to $N+2$ decoupled eigenvalue problems
\[
  {\mathbb A}_j\widetilde{\widetilde{\boldsymbol{y}}} = \lambda
  \widetilde{\widetilde{\boldsymbol{y}}},\quad j=0,1,\ldots, N+1,
\]
where,
\[
  {\mathbb A}_j = M^{-1/2} (K^{(1)}_j)^{-1/2} \hat{P}^T \hat{M} \hat{K}^{(1)}_j \hat{P} (K^{(1)}_j)^{-1/2} M^{-1/2}.
\]
In this case, the diagonal matrices $K^{(1)}_j$ and $\hat{K}^{(1)}_j$ are the blocks of the diagonal matrices
\[
J_1 g_1^{11} = \begin{bmatrix}
  K^{(1)}_0 & & & \\
  & K^{(1)}_1 & & \\
  & & \ddots & \\
  & & & K^{(1)}_{N+1}
\end{bmatrix}
,\quad
\hat{J}_1 \hat{g}_1^{11} = \begin{bmatrix}
  \hat{K}^{(1)}_0 & & & \\
  & \hat{K}^{(1)}_1 & & \\
  & & \ddots & \\
  & & & \hat{K}^{(1)}_{N+1}
\end{bmatrix}.
\]
Here, $K^{(1)}_j$ and $\hat{K}^{(1)}_j$ hold the values of the corresponding metric coefficients
along a nodal or cell-centered grid line in the $r^1$-direction with $r^2=\hat{x}_j$.
\end{appendix}

\bibliographystyle{plain}
\bibliography{references}

\end{document}

%% file: macros.tex
\newcommand{\xib}{{\boldsymbol \xi}}
\newcommand{\ab}{{\mathbf a}}

\newcommand{\eb}{{\mathbf e}}
\newcommand{\fb}{{\mathbf f}}

\newcommand{\pb}{{\mathbf p}}

\newcommand{\rb}{{\mathbf r}}
\newcommand{\yb}{{\mathbf y}}

\newcommand{\qb}{{\mathbf q}}
\newcommand{\ub}{{\mathbf u}}
\newcommand{\vb}{{\mathbf v}}
\newcommand{\wb}{{\mathbf w}}

\newcommand{\oneb}{{\mathbf 1}}

\newcommand{\p}{\partial}

\newcommand{\ba}{\begin{array}}
\newcommand{\ea}{\end{array}}
\newcommand{\be}{\begin{equation}}
\newcommand{\ee}{\end{equation}}
\newcommand{\bd}{\begin{displaymath}}
\newcommand{\ed}{\end{displaymath}}

\newcommand{\xb}{\mathbf{x}}

\newcommand{\Sb}{\mathbf S}

\newcommand{\usecomments}{true}
\newcommand{\ocomment}[1] {
\ifthenelse{ \equal{\usecomments}{true} }{
 \textbf{Ossian: }{\color{blue} #1}
}{
}
}

\definecolor{darkgreen}{rgb}{0.0, 0.588, 0.178}
\newcommand{\revisedblue}[1]{{#1}}
\newcommand{\revisedred}[1]{{#1}}
\newcommand{\revisedgreen}[1]{{#1}}